\documentclass[12pt,reqno]{amsart}
\usepackage{tikz}
\usepackage{url}
\tikzset{empty/.style={black, fill=white}}

\pgfdeclarelayer{background}
\pgfdeclarelayer{foreground}
\pgfsetlayers{background,main,foreground}

\usepackage{graphics}

\usepackage{tikzit}

\tikzstyle{empty}=[fill=none, draw=black, shape=circle, tikzit fill=white]
\tikzstyle{full}=[fill=black, draw=black, shape=circle, tikzit shape=circle]
\tikzstyle{red}=[fill=none, draw=none, shape=circle, tikzit draw=black, tikzit fill=red]
\tikzstyle{pink}=[fill=white, draw=black, shape=circle, tikzit draw=black, tikzit fill={rgb,255: red,160; green,6; blue,255}]
\tikzstyle{green}=[fill=white, draw=black, shape=circle, tikzit fill={rgb,255: red,0; green,181; blue,12}, tikzit draw=black]
\tikzstyle{orange}=[fill=white, draw=black, shape=circle, tikzit fill={rgb,255: red,255; green,174; blue,34}, tikzit draw=black]
\tikzstyle{blue}=[fill={rgb,255: red,64; green,0; blue,255}, draw=black, shape=circle]

\tikzstyle{edgeD}=[-, dashed=dashed, draw=black, fill=none]
\tikzstyle{edgeGreen}=[-, draw={rgb,255: red,0; green,181; blue,12}]
\tikzstyle{edgeOrange}=[-, draw={rgb,255: red,255; green,174; blue,34}]
\tikzstyle{edgeRed}=[-, draw=red]

\newtheorem{Theorem}{Theorem}[section]
\newtheorem{Lemma}[Theorem]{Lemma}
\newtheorem{Corollary}[Theorem]{Corollary}
\newtheorem{Proposition}[Theorem]{Proposition}
\newtheorem{Remark}[Theorem]{Remark}

\newtheorem{Example}[Theorem]{Example}

\newtheorem{Definition}[Theorem]{Definition}

\newtheorem{Conjecture}[Theorem]{Conjecture}

\def\qed{\ifhmode\textqed\fi
	\ifmmode\ifinner\hfill\quad\qedsymbol\else\dispqed\fi\fi}
\def\textqed{\unskip\nobreak\penalty50
	\hskip2em\hbox{}\nobreak\hfill\qedsymbol
	\parfillskip=0pt \finalhyphendemerits=0}
\def\dispqed{\rlap{\qquad\qedsymbol}}

\def\reg{\textup{reg}}
\def\pd{\textup{proj\,dim}}
\def\depth{\textup{depth}}
\def\FF{\mathbb{F}}

\begin{document}

	\title{The size of the betti table of binomial edge ideals}
	\author{Antonino Ficarra, Emanuele Sgroi}
	
	\address{Antonino Ficarra, Department of mathematics and computer sciences, physics and earth sciences, University of Messina, Viale Ferdinando Stagno d'Alcontres 31, 98166 Messina, Italy}
	\email{antficarra@unime.it}
	
	\address{Emanuele Sgroi, Department of mathematics and computer sciences, physics and earth sciences, University of Messina, Viale Ferdinando Stagno d'Alcontres 31, 98166 Messina, Italy}
	\email{emasgroi@unime.it}
	
	\thanks{.
	}
	
	
	
	\maketitle
	
	\begin{abstract}
		Let $G$ be a finite simple graph on $n$ non-isolated vertices, and let $J_G$ be its binomial edge ideal. We determine almost all pairs $(\pd(J_G),\reg(J_G))$, where $G$ ranges over all finite simple graphs on $n$ non-isolated vertices, for any $n$.
	\end{abstract}
	
	\section*{Introduction}
	One never-ending source of inspiration in Combinatorial Commutative Algebra is the study of the minimal free resolutions of graded ideals. Let $R=K[x_1,\dots,x_n]$ be the standard graded polynomial ring over a field $K$, and let $I\subset R$ be a graded ideal. The behaviour of the minimal resolution of $I$ is hard to predict. Two important homological invariants of $I$, that provide a measure of the complexity of its minimal resolution, are the \textit{projective dimension}, $\pd(I)$, and the \textit{regularity}, $\reg(I)$. The pair $(\pd(I),\reg(I))$ determines the size of the \textit{Betti table} of $I$.
	
	A central question is the following. For a given class $\mathcal{C}$ of graded ideals of $R$, can we determine the set of the sizes of the Betti tables of the ideals in $\mathcal{C}$? That is, can we determine all pairs $(\pd(I),\reg(I))$, $I\in\mathcal{C}$? Such a problem is very difficult, and the behaviour of these pairs is quite mysterious. On the other hand, if the graded ideals in $\mathcal{C}$ arise from combinatorics, then their combinatorial nature helps us to better understand and sometimes also answer such a question.
	
	A problem of the type discussed above is considered in \cite{HH21}. Hereafter, by a graph $G$ we mean a finite simple graph. Recall that the \textit{edge ideal} $I(G)$ of $G$ is the ideal generated by the monomials $x_ix_j$ where $\{i,j\}$ is an edge of $G$. In \cite{HH21}, H\`a and Hibi studied the question of determining all admissible pairs $(\pd(I(G)),\reg(I(G)))$, as $G$ ranges over all graphs on a given number of vertices. This question is related to the search of a \textsc{max min vertex cover} and a \textsc{min max independent set}. These classical problems of graph theory are known to be NP-hard, and they received a lot of attention lately \cite{BDCP15,BDCEP13,CHLN07,H93}. On the other hand, in \cite{HH21} the combinatorics of $G$ yields the following surprising lower bound: $\pd(I(G))\ge 2\sqrt{n}-3$. The authors determined all pairs $(\pd(I(G)),\reg(I(G)))$ when the projective dimension reaches this lower bound, and also when the regularity reaches its minimal possible value, namely $\reg(I(G))=2$. For the class of connected bipartite graphs, the H\`a--Hibi problem was completely solved by Erey and Hibi in \cite{EH20}. Similar questions are treated in \cite{HKKMV21,HKM19,HKKMT21,HKMV20,HM21,HKU23,KS20} and the references therein.
	
	Another family of graded ideals arising from graphs is that of \textit{binomial edge ideals}. In 2010, Herzog, Hibi, Hreinsdóttir, Kahle and Rauh in \cite{HHHKR10}, and independently, Ohtani in \cite{O11}, introduced the \textit{binomial edge ideal}. Let $G$ be a graph on $n$ vertices $1,\dots,n$. Then $J_G$ is defined to be the graded ideal of $S=K[x_1,\dots,x_n,y_1,\dots,y_n]$ generated by the binomials $x_iy_j-x_jy_i$ for all edges $\{i,j\}$ of $G$. This class of ideals generalizes the classical determinantal ideals. Indeed, $J_G$ may be seen as the ideal generated by an arbitrary set of maximal minors of a $(2\times n)$-matrix of indeterminates. A huge effort has been made to understand the homological properties of binomial edge ideals. Consult the surveys \cite{MD22,M-Survey} for the current state of art.
	
	In this article, we address the H\`a--Hibi problem for the class of binomial edge ideals. One would guess, as in the case of edge ideals, that this problem is difficult, and that answering this question in an explicit fashion may not be possible. Quite surprisingly, we succeed to deliver a fairly comprehensive and explicit solution.
	
	Let $n\ge1$ be an integer. Denote by $\textup{Graphs}(n)$ the class of all finite simple graphs on $n$ non-isolated vertices. Then, we define
	$$
	\textup{pdreg}(n)\ =\ \big\{(\pd(J_G),\reg(J_G)):G\in\textup{Graphs}(n)\big\},
	$$
	which is the set of the sizes of the Betti tables of $J_G$, as $G$ ranges over all graphs on $n$ non-isolated vertices. Our main result in this article is the following theorem.\medskip\\
	\textbf{Theorem \ref{Thm:pdreg(n)}} \textit{For all $n\ge3$,
	\begin{equation*}
	\begin{aligned}
	\textup{pdreg}(n)\ =\ \big\{(n-2,2),(n-2,n)\big\}\cup\bigcup_{r=3}^{\lfloor\frac{n}{2}\rfloor+1}\big(\bigcup_{p=n-r}^{2n-5}\{(p,r)\}\big)\ \cup\ \\
	\cup \bigcup_{r=\lceil\frac{n}{2}\rceil+1}^{n-2}\big(\bigcup_{p=r-2}^{2n-5}\{(p,r)\}\big)\cup A_n,
	\end{aligned}
	\end{equation*}
	where $A_n=\{(p,r)\in\textup{pdreg}(n):r=n-1\}$.}\medskip
	
	The reader may see that when the regularity is $n-1$ we leave the set $A_n$ not determined. Indeed, our experiments show a quite unexpected behaviour.\medskip\\
	\textbf{Conjecture \ref{Conj:reg=n-1}} Let $G$ be a graph on $n\ge7$ non-isolated vertices. Suppose that $\reg(J_G)=n-1$. Then $\pd(J_G)\le n$.\medskip
	
	The article is structured as follows. In Section \ref{FS22:sec1}, we state some general bounds for the projective dimension and the regularity of binomial edge ideals. The proof of Theorem \ref{Thm:pdreg(n)} is by induction on the number of vertices of the graph. On the other hand, there are some special graphs giving some of the pairs $(p,r)\in\textup{pdreg}(n)$ that we did not obtain by inductive arguments. In Section \ref{FS22:sec2}, we discuss these special classes of graphs. They give the pairs $(p,r)\in\textup{pdreg}(n)$ with $p=2n-5,2n-6$, or $r=3,n-2$. Section \ref{FS22:sec3} contains the main result in the article. Our answer is nearly complete. Indeed, only for the graphs with almost maximal regularity, namely $\reg(J_G)=n-1$, we do not know yet the projective dimension. It would be interesting to classify the binomial edge ideals with almost maximal regularity.
	
	We gratefully acknowledge the use of \textit{Macaulay2} \cite{GDS} and in particular of the package \texttt{NautyGraphs} \cite{MPNauty}.

	\section{General bounds for the betti table of binomial edge ideals}\label{FS22:sec1}
	
	Let $I$ be a graded ideal of a standard graded polynomial ring $R=K[x_1,\dots,x_n]$, where $K$ is a field. Then $I$ possesses a unique minimal graded free resolution
	$$
	\FF\ \ :\ \ \cdots\rightarrow F_i\rightarrow F_{i-1}\rightarrow\dots\rightarrow F_1\rightarrow F_0\rightarrow I\rightarrow0,
	$$
	with $F_i=\bigoplus_jR(-j)^{\beta_{i,j}(I)}$, where the $\beta_{i,j}(I)$ are the \textit{graded Betti numbers} of $I$. The \textit{projective dimension} and the \textit{regularity} of $I$, are, respectively,
	\begin{align*}
	\pd(I)\ &=\ \max\{i:\beta_{i,j}(I)\ne0,\ \text{for some}\ j\},\\
	\reg(I)\ &=\ \max\{j-i:\beta_{i,j}(I)\ne0,\ \text{for some}\ i\ \text{and}\ j\}.
	\end{align*}
	
	Throughout the article, we consider only finite simple graphs. Hence we will refer to them simply as graphs. Let $G$ be a graph. By $V(G)$ we denote the \textit{vertex set} of $G$, and by $E(G)$ the \textit{edge set} of $G$. Two vertices $u$ and $v$ are \textit{adjacent} if $\{u,v\}\in E(G)$. A vertex $u$ is called \textit{isolated} if $\{u,v\}\notin E(G)$ for all $v\in V(G)\setminus\{u\}$.
	
	Let $K$ be a field and $G$ be a graph with vertex set $\{1,\dots,n\}$. The \textit{binomial edge ideal} $J_G$ of $G$ is the following binomial ideal of $S=K[x_1,\dots,x_n,y_1,\dots,y_n]$:
	$$
	J_G=(x_iy_j-x_jy_i\ :\ \{i,j\}\in E(G)).
	$$
	
	Let $G$ be a graph. If $W$ is a subset of $V(G)$ we denote by $G_W$ the \textit{induced subgraph of $G$ on $W$}, that is $V(G_W)=W$ and $E(G_W)=\{\{u,v\}\in E(G):u,v\in W\}$. It is known by \cite[Corollary 2.2]{MM13} that
	$$
	\pd(J_G)\ge\pd(J_{G_W})\ \ \text{and}\ \ \reg(J_G)\ge\reg(J_{G_W}),
	$$
	for all subsets $W$ of $V(G)$. We will use freely this fact.
	
	Suppose that $G$ is connected. Then, we say that $G$ is \textit{$\ell$-vertex-connected} if for all subsets $W$ of $V(G)$ with $|W|<\ell$, the induced subgraph $G_{V(G)\setminus W}$ is connected. The \textit{vertex-connectivity} of $G$, denoted by $\ell(G)$, is the maximum integer $\ell$ such that $G$ is $\ell$-vertex-connected. It is clear that $\ell(G)\ge1$.
	
	We denote by $K_n$ the \textit{complete graph} on $n$ vertices, that is $V(K_n)=\{1,\dots,n\}$ and $\{i,j\}\in E(K_n)$ for all $i,j\in V(K_n)$, $i\ne j$. Whereas, by $P_n$ we denote the \textit{path} of length $n-1$, that is $V(P_n)=\{1,\dots,n\}$ and $E(P_n)=\{\{1,2\},\{2,3\},\ldots,\{n-1,n\}\}$.
		
	\begin{Theorem}\label{Thm:pdGconnect}
		Let $G$ be a graph on $n\ge3$ non-isolated vertices. Then
		$$
		\pd(J_G)\le 2n-5.
		$$
		Furthermore, if $G$ is connected on $n\ge2$ vertices, then $\pd(J_G)\ge n-2$.
	\end{Theorem}
	\begin{proof}
		By the work of \cite[Theorem 5.2]{MMK21a} it is known that $\depth(S/J_G)\ge4$ for all graphs $G$. Therefore, by the Auslander--Buchsbaum formula we have
		$$
		\pd(J_G)=\pd(S/J_G)-1=2n-\depth(S/J_G)-1\le 2n-5.
		$$
		Suppose that $G$ is connected and not complete. By \cite[Theorems 3.19 and 3.20]{BNB17} we have $\pd(J_G)=\pd(S/J_G)-1\ge n+\ell(G)-3$. Since $\ell(G)\ge1$, we obtain $\pd(J_G)\ge n-2$. Else, if $G=K_n$ is the complete graph, then $\pd(J_G)=n-2$ because the Eagon-Northcott complex is the minimal free resolution of $J_{K_n}$.
	\end{proof}
	
	\begin{Theorem}\label{Thm:regGbounds}
		Let $G$ be a graph on $n\ge2$ non-isolated vertices. Then
		$$
		2\le\reg(J_G)\le n.
		$$
		Moreover,
		\begin{enumerate}
			\item[\textup{(i)}] $\reg(J_G)=2$ if and only if $G=K_n$ and, in this case, $\pd(J_{K_n})=n-2$.
			\item[\textup{(ii)}] $\reg(J_G)=n$ if and only if $G=P_n$ and, in this case, $\pd(J_{P_n})=n-2$.
		\end{enumerate}
	\end{Theorem}
	\begin{proof}
		Since $J_G$ is generated in degree two, $\reg(J_G)\ge2$. By \cite[Theorem 1.1]{MM13}, $\reg(J_G)\le|V(G)|=n$. Statement (i) follows from \cite[Theorem 2.1]{MK12}. Statement (ii) follows from \cite[Theorem 3.2]{KM16} and \cite[Theorem 1.1]{MM13}.
	\end{proof}
	
	The following observation will be used several times.
	\begin{Remark}\label{Rem:Gdecomp}
		\rm Suppose $G=G_1\sqcup G_2\sqcup\dots\sqcup G_c$ is a graph without isolated vertices and $c$ connected components $G_i$, $i=1,\dots,c$. Here $\sqcup$ denotes the disjoint union of the graphs $G_i$. Let $S_i=K[x_v,y_v:v\in V(G_i)]$ and $n_i=|V(G_i)|$, for all $i$. Then $\sum_{i=1}^cn_i=n$ and $J_{G_i}$ is a binomial ideal of $S_i$. Since the polynomial rings $S_i$ are in pairwise disjoint sets of variables, we have
		$$
		S/J_G\cong\bigotimes_{i=1}^cS_i/J_{G_i}.
		$$
		Hence, $\pd(S/J_G)=\sum_{i=1}^c\pd(S_i/J_{G_i})$ and $\reg(S/J_G)=\sum_{i=1}^c\reg(S_i/J_{G_i})$. Taking into account that for a graded ideal $I$ of a polynomial ring $R$ we have $\pd(R/I)=\pd(I)+1$ and $\reg(R/I)=\reg(I)-1$, we obtain the following useful identities,
		\begin{align}
		\label{eq:pdJGc}\pd(J_G)\ &=\ \sum_{i=1}^c\pd(J_{G_i})+(c-1),\\
		\label{eq:regJGc}\reg(J_G)\ &=\ \sum_{i=1}^c\reg(J_{G_i})-(c-1).
		\end{align}
	\end{Remark}
	
	Next we provide a different lower bound for the projective dimension of a binomial edge ideal, in terms of the regularity.
	\begin{Proposition}\label{Prop:refinedPd1}\label{Prop:refinedPd2}
		Let $n\ge3$, $r$ be positive integers with $3\le r\le n-1$. Then
		$$
		\max\{n-r,r-2\}\le\pd(J_G)\le 2n-5,
		$$
		for all graphs $G$ on $n$ non-isolated vertices such that $\reg(J_G)=r$.
	\end{Proposition}
	\begin{proof}
		The upper bound for $\pd(J_G)$ is stated in Theorem \ref{Thm:pdGconnect}. To prove the lower bound, assume the notation of Remark \ref{Rem:Gdecomp}. Notice that
		$$
		\max\{n-r,r-2\}\ =\ \begin{cases}
			n-r,&\textup{if}\,\ 3\le r\le\lfloor\frac{n}{2}\rfloor+1,\\
			r-2,&\textup{if}\,\ \lceil\frac{n}{2}\rceil+1\le r\le n-1.
		\end{cases}
		$$
		
		In view of this computation, we distinguish the two possible cases.
		
		Suppose $\max\{n-r,r-2\}=n-r$. By formula (\ref{eq:regJGc}) and Theorem \ref{Thm:regGbounds},
		\begin{equation}\label{eq:laterused}
			\begin{aligned}
		r\ &=\ \sum_{i=1}^c\reg(J_{G_i})-(c-1)\ge\sum_{i=1}^c2-(c-1)\\&=\ 2c-c+1=c+1.
		\end{aligned}
		\end{equation}
		Therefore, $c\le r-1$. Now, by formula (\ref{eq:pdJGc}) and Theorem \ref{Thm:pdGconnect},
		\begin{equation}\label{eq:ProjDimCconnectedG}
		\begin{aligned}
		\pd(J_G)\ &=\ \sum_{i=1}^c\pd(J_{G_i})+(c-1)\\&\ge\ \sum_{i=1}^c(n_i-2)+(c-1)\\&=\ \sum_{i=1}^cn_i-2c+c-1\\&=\ n-c-1\ge n-r,
		\end{aligned}
		\end{equation}
		since $c\le r-1$.
		
		Suppose now $\max\{n-r,r-2\}=r-2$. By formula (\ref{eq:regJGc}) and Theorem \ref{Thm:regGbounds},
		\begin{align*}
		r\ &=\ \sum_{i=1}^c\reg(J_{G_i})-(c-1)\le\sum_{i=1}^cn_i-(c-1)\\&=\ n-c+1.
		\end{align*}
		Thus, $c\le n-r+1$. Using the computation in (\ref{eq:ProjDimCconnectedG}), we obtain
		\begin{align*}
		\pd(J_G)\ &\ge\ n-c-1\ge r-2,
		\end{align*}
		since $c\le n-r+1$.
	\end{proof}
	
	\section{Special classes of graphs}\label{FS22:sec2}
	
	In this section, we determine some classes of graphs that have a given projective dimension or a given regularity. These families will be used to get our main result.
	
	Let $G_1,G_2$ be two graphs with disjoint vertex sets $V_1$ and $V_2$ and edge sets $E_1$ and $E_2$. The \textit{join} of $G_1$ and $G_2$ is defined to be the graph $G_1*G_2$ with vertex set $V_1\cup V_2$ and edge set $E_1\cup E_2\cup\{\{v_1,v_2\}:v_1\in V_1,v_2\in V_2\}$.
	
	The following formula due to Madani and Kiani (\cite[Theorem 2.1]{MK18}) plays a pivotal role in our article. Suppose $G_1$ and $G_2$ are graphs with disjoint vertex sets $V_1$ and $V_2$, and that not both of them are complete. Then,
	\begin{equation}\label{eq:RegJoinGraphs}
	\reg(J_{G_1*G_2})=\max\{\reg(J_{G_1}),\reg(J_{G_2}),3\}.
	\end{equation}
	
	It is worth mentioning that, even if $G_1$ and $G_2$ may have isolated vertices, all vertices in $G_1*G_2$ are non-isolated, and $G_1*G_2$ is always a connected graph. Moreover, if $G_1\ne K_1$ and $G_2=K_1$, then $G_1*G_2$ is called a \textit{cone}.
	
	\subsection{Graphs with projective dimension $2n-5$}
	
	In \cite[Theorem 5.3]{MMK21a}, Malayeri, Madani and Kiani have characterized all graphs $G$ that have minimal depth possible, \emph{i.e.}, $\depth(S/J_G)=4$. For such graphs, by the Auslander--Buchsbaum formula we have maximal projective dimension:
	$$\pd(J_G)=\pd(S/J_G)-1=2n-\depth(S/J_G)-1=2n-5.$$ 
	
	Hereafter, for a positive integer $m\ge1$, we denote by $mK_1$ a graph consisting of $m$ isolated vertices.
	
	\begin{Theorem}\label{Thm:pd=2n-5}
		Let $G$ be a graph on $n\ge5$ non-isolated vertices. Then, the following conditions are equivalent.
		\begin{enumerate}
			\item[\textup{(i)}] $\pd(J_G)=2n-5$.
			\item[\textup{(ii)}] $G=\widetilde{G}*2K_1$ for some graph $\widetilde{G}$ on $n-2$ vertices.
		\end{enumerate}
	\end{Theorem}
	
	\begin{Corollary}\label{Cor:pd=2n-5}
		Let $G$ be a graph on $n\ge5$ non-isolated vertices with $\pd(J_G)=2n-5$. Then,
		$$
		3\le\reg(J_G)\le n-2.
		$$
		Furthermore, for any integers $n\ge 5$ and $r\in\{3,\dots,n-2\}$, there exists a graph $G$ on $n$ non-isolated vertices such that $\pd(J_G)=2n-5$ and $\reg(J_G)=r$.
	\end{Corollary}
	\begin{proof}
		By the previous theorem, $G=\widetilde{G}*2K_1$ for some graph $\widetilde{G}$ on $n-2$ vertices. Since $2K_1$ is not complete and $J_{2K_1}=(0)$, by formula (\ref{eq:RegJoinGraphs}),
		$$
		\reg(J_G)=\max\{\reg(J_{\widetilde{G}}),\reg(J_{2K_1}),3\}=\max\{\reg(J_{\widetilde{G}}),3\}.
		$$
		Hence $\reg(J_G)\ge3$. Moreover, by Theorem \ref{Thm:regGbounds}, $\reg(J_{\widetilde{G}})\le|V(\widetilde{G})|=n-2$ and so $\reg(J_G)\le n-2$, since $n-2\ge3$.
		
		Now, let $r\in\{3,\dots,n-2\}$. Set $\widetilde{G}=P_r\sqcup(n-2-r)K_1$ and $G=\widetilde{G}*2K_1$. Then $\pd(J_G)=2n-5$, by the previous theorem. Moreover,
		$$
		\reg(J_G)=\max\{\reg(J_{\widetilde{G}}),\reg(J_{2K_1}),3\}=\max\{r,3\}=r,
		$$
		since $J_{\widetilde{G}}=J_{P_r}$ has regularity $r$ by Theorem \ref{Thm:regGbounds}(ii).
	\end{proof}
	
	\subsection{Graphs with projective dimension $2n-6$}
	After the case of minimal depth, Malayeri, Madani and Kiani classified in \cite[Theorem 5, Section 5]{MMK22} the graphs $G$ with $\depth(S/J_G)=5$, that is $\pd(J_G)=2n-6$.
	
	To state their result, we introduce the following class of graphs. Hereafter, if $n$ is a positive integer, we denote by $[n]$ the set $\{1,2,\dots,n\}$. If $v$ is a vertex of a graph $G$, $N_G(v)=\{u\in V(G)\setminus\{v\}:\{u,v\}\in E(G)\}$ is the \textit{neighbourhood of $v$ in $G$}.
	\begin{Definition}\label{Def:defG_T}
		\rm Let $T\subset[n]$ with $|T|=n-2$. The family $\mathcal{G}_T$ consists of all graphs $G$ with vertex set $[n]$ such that there exist two non-adjacent vertices $u$ and $v$ of $G$ with $u,v\in[n]\setminus T$, and three disjoint subsets of $T$, say $V_0$, $V_1$ and $V_2$ with $V_1,V_2\ne\emptyset$ and $V_0\cup V_1\cup V_2=T$, such that the following conditions hold:
		\begin{enumerate}
			\item[\textup{(a)}] $N_G(u)=V_0\cup V_1$ and $N_G(v)=V_0\cup V_2$
			\item[\textup{(b)}] $\{v_1,v_2\}\in E(G)$, for every $v_1\in V_1$ and $v_2\in V_2$.
		\end{enumerate}
	\end{Definition}

	Now, we introduce the class of $D_5$-type graphs \cite[Definition 4]{MMK22}.
	\begin{Definition}
		\rm Let $G$ be a graph on $n$ vertices such that $G\ne H*2K_1$, for all graphs $H$. The graph $G$ is a \textit{$D_5$-type graph} if one of the following conditions holds:
		\begin{enumerate}
			\item[\textup{(a)}] $G=\widetilde{G}*3K_1$, for some graph $\widetilde{G}$.
			\item[\textup{(b)}] $G=\widetilde{G}*(K_1\sqcup K_2)$, for some graph $\widetilde{G}$.
			\item[\textup{(c)}] $G\in\mathcal{G}_T$, for some $T\subset V(G)$ with $|T|=n-2$.
		\end{enumerate}
	\end{Definition}
	
	\begin{Theorem}\label{Thm:pd=2n-6}
		\textup{(\cite[Theorem 5]{MMK22}).} Let $G$ be a graph on $n\ge5$ non-isolated vertices. Then, the following conditions are equivalent.
		\begin{enumerate}
			\item[\textup{(i)}] $\pd(J_G)=2n-6$.
			\item[\textup{(ii)}] $G$ is a $D_5$-type graph.
		\end{enumerate}
	\end{Theorem}
	
	The following observation will be useful later.
	\begin{Remark}\label{Rem:pd=2n-5,2n-6,Connected}
		\rm Let $n\ge5$. By Theorems \ref{Thm:pd=2n-5} and \ref{Thm:pd=2n-6}, it follows that all graphs $G$ on $n\ge5$ non-isolated vertices such that $\pd(J_G)=2n-5$ or $\pd(J_G)=2n-6$ are connected.
	\end{Remark}
	
	For the proof of the next result, we need a lemma of Ohtani \cite[Lemma 4.8]{O11}, see also \cite[Lemma 3.1]{KS20} and formula (1) in the same article.
	
	We recall that a \textit{clique} of a graph $G$ is a subset $W$ of $V(G)$ such that $G_W$ is an induced complete subgraph of $G$. A \textit{maximal clique} of $G$ is a clique of $G$ that is not contained in any other clique of $G$. A vertex $v\in V(G)$ is called \textit{simplicial} if $N_G(v)$ is a clique of $G$, otherwise is called \textit{internal}. Let $v\in V(G)$. We denote $G_{V(G)\setminus\{v\}}$ by $G\setminus v$. Whereas, by $G_v$ we denote the graph with vertex set $V(G)$ and edge set $E(G)\cup\{\{w_1,w_2\}:w_1,w_2\in N_G(v)\}$.
	
	Let $v$ be an internal vertex of a graph $G$. Then Ohtani lemma, see also formula (1) in \cite{KS20}, implies that the following short sequence is exact:
	\begin{equation}\label{eq:ShortExactSeqOhtani}
	0\rightarrow\frac{S}{J_G}\longrightarrow\frac{S}{(x_v,y_v,J_{G\setminus v})}\oplus \frac{S}{J_{G_v}}\longrightarrow\frac{S}{(x_v,y_v,J_{G_v\setminus v})}\rightarrow0.
	\end{equation}
	
	In the next proposition we also use freely the following upper bound for the regularity proved in \cite[Theorem 2.1]{ERT20}, see also formula (3) in the same article. For a connected graph $G$ on $n$ vertices,
	$$
	\reg(J_G)\le n+2-|W|,\ \ \ \text{for any clique}\ W\ \text{of}\ G.
	$$
	
	\begin{Proposition}\label{Prop:pd=2n-6}
		Let $G$ be a graph on $n\ge6$ non-isolated vertices such that $\pd(J_G)=2n-6$. Then,
		$$
		3\le\reg(J_G)\le n-2.
		$$
	\end{Proposition}
	\begin{proof}
		By the previous theorem, $G$ is a $D_5$-type graph. Therefore $G$ is not a complete graph. By Theorem \ref{Thm:regGbounds}(i), $\reg(J_G)\ge3$.
		
		Let us prove the upper bound. Firstly, suppose $G=\widetilde{G}*3K_1$ or $G=\widetilde{G}*(K_1\sqcup K_2)$, for some graph $\widetilde{G}$ on $n-3$ vertices. Then, arguing as in Corollary \ref{Cor:pd=2n-5} we obtain $\reg(J_G)\le n-3$ in this case.\smallskip
		
		Suppose now that $G\in\mathcal{G}_T$ for some $T\subset V(G)$ with $|T|=n-2$. Then $V(G)=\{u,v\}\cup V_0\cup V_1\cup V_2$ where the union is disjoint, $V_1,V_2\ne\emptyset$, $u$ and $v$ are non-adjacent, $N_G(u)=V_0\cup V_1$, $N_G(v)=V_0\cup V_2$ and $\{v_1,v_2\}\in E(G)$ for all $v_1\in V_1$ and $v_2\in V_2$.\smallskip
		
		Let us show that $\reg(J_G)\le n-2$. We distinguish three cases.\smallskip\\
		\textbf{Case 1.} Suppose both $u$ and $v$ are simplicial vertices of $G$. Then, $T=V_0\cup V_1\cup V_2$ is a maximal clique of $G$. Therefore, $\reg(J_G)\le n+2-|T|=4\le n-2$ as $n\ge6$.\smallskip
		
		For the proof of the next two cases, we note that $G\setminus v$ is not a path. Assume for a contradiction, $G\setminus v$ is a path. Then every vertex of $G\setminus v$ is adjacent to at most two vertices of $G\setminus v$. Let $v_1\in V_1$ and $v_2\in V_2$. Since $v_1$ is adjacent to $u$ and $v_2$, it follows that $V_2=\{v_2\}$ as otherwise $v_1$ would have more than two neighbors in $G\setminus v$. Then $|V_0\cup V_1|=n-3\ge 3$ which implies $u$ has at least three neighbors in $G\setminus v$, which is a contradiction.\smallskip\\
		\textbf{Case 2.} Suppose that $u$ is a simplicial vertex, but $v$ is internal. Then $\{u\}\cup V_0\cup V_1$ is a clique of $G$. Since $v$ is internal, we can apply Ohtani lemma. By the short exact sequence (\ref{eq:ShortExactSeqOhtani}) we have
		$$
		\reg(J_G)\le\max\{\reg(x_v,y_v,J_{G\setminus v}),\reg(J_{G_v}),\reg(x_v,y_v,J_{G_v\setminus v})+1\}.
		$$
		Since $x_v,y_v$ do not divide any generator of $J_{G\setminus v}$ and $J_{G_v\setminus v}$, we obtain
		\begin{equation}\label{eq:bound|V_0|1}
		\reg(J_G)\le\max\{\reg(J_{G\setminus v}),\reg(J_{G_v}),\reg(J_{G_v\setminus v})+1\}.
		\end{equation}
	    Since $G\setminus v$ is not a path, by Theorem \ref{Thm:regGbounds}(ii) we have
		\begin{equation}\label{eq:bound>|V_0|2}
		\reg(J_{G\setminus v})\ \le\ |V(G\setminus v)|-1=n-2.
		\end{equation}
		Note that in the graphs $G_v$ and $G_v\setminus v$, the set $T=V_0\cup V_1\cup V_2$ is a clique of size $n-2$. Therefore, since $n\ge6$ we have
		\begin{align}
		\label{eq:bound|V_0|3}\reg(J_{G_v})\ &\le\ n+2-|T|=4\le n-2,\\
		\label{eq:bound|V_0|4}\reg(J_{G_v\setminus v})+1\ &\le\ (n-1)+2-|T|+1=4\le n-2.
		\end{align}
		Combining (\ref{eq:bound|V_0|1}) with (\ref{eq:bound>|V_0|2}), (\ref{eq:bound|V_0|3}) and (\ref{eq:bound|V_0|4}) we obtain $\reg(J_G)\le n-2$, as desired.\medskip\\
		\textbf{Case 3.} Suppose that both $u$ and $v$ are internal vertices. We first apply Ohtani lemma to $v$ and get the inequality,
		\begin{equation}\label{eq:bound|V_0|1p1}
		\reg(J_G)\le\max\{\reg(J_{G\setminus v}),\reg(J_{G_v}),\reg(J_{G_v\setminus v})+1\}.
		\end{equation}
		Since $G\setminus v$ is not a path, by Theorem \ref{Thm:regGbounds}(ii) we have
		\begin{equation}\label{eq:bound|V_0|1p6}
		\reg(J_{G\setminus v})\le n-2.
		\end{equation}
		
		Since $u$ and $v$ are non-adjacent in $G$ and in $G_v$, $u$ is internal (simplicial) in $G_v$ if and only if $u$ is internal (simplicial) in $G_v\setminus v$. We distinguish the two cases.\smallskip\\
		\textbf{Subcase 3.1.} Suppose $u$ is simplicial in $G_v$ and $G_v\setminus v$. Then $T=V_0\cup V_1\cup V_2$ is a clique of size $n-2$ in both graphs. Therefore,
		\begin{align}
		\label{eq:bound|V_0|1p7}\reg(J_{G_v})\ &\le\ n+2-|T|=4\le n-2,\\
		\label{eq:bound|V_0|1p8}\reg(J_{G_v\setminus v})+1\ &\le\ (n-1)+2-|T|+1=4\le n-2.
		\end{align}
		Combining (\ref{eq:bound|V_0|1p1}) with (\ref{eq:bound|V_0|1p6}), (\ref{eq:bound|V_0|1p7}) and (\ref{eq:bound|V_0|1p8}), we obtain $\reg(J_{G})\le n-2$, as desired.\smallskip\\
		\textbf{Subcase 3.2.} Suppose $u$ is internal in $G_v$ and $G_v\setminus v$. We apply Ohtani lemma to get
		\begin{align}
		\label{eq:bound|V_0|1p9}\reg(J_{G_v})\ &\le\ \max\{\reg(J_{G_v\setminus u}),\reg(J_{(G_v)_u}),\reg(J_{(G_v)_u\setminus u})+1\},\\
		\label{eq:bound|V_0|1p10}\reg(J_{G_v\setminus v})\ &\le\ \max\{\reg(J_{G_v\setminus\{u,v\}}),\reg(J_{(G_v\setminus v)_u}),\reg(J_{(G_v\setminus v)_u\setminus u})+1\}.
		\end{align}
		The graph $G_v\setminus u$ is easily seen to be not a path. Hence, $\reg(J_{G_v\setminus u})\le|V(G_v\setminus u)|-1=n-2$. In $(G_v)_u$ the set $T$ is a clique of size $n-2$, and the same calculation as in (\ref{eq:bound|V_0|1p7}) gives $\reg(J_{(G_v)_u})\le n-2$. Finally, in $(G_v)_u\setminus u$, $T$ is a clique. The same calculation as in (\ref{eq:bound|V_0|1p8}) gives $\reg(J_{(G_v)_u\setminus u})+1\le n-2$. These calculations and equation (\ref{eq:bound|V_0|1p9}) yield
		\begin{equation}\label{eq:bound|V_0|1p11}
		\reg(J_{G_v})\le n-2.
		\end{equation}
		Note that $G_v\setminus\{u,v\}$ is not a path. Indeed, $|T|=|V_0\cup V_1\cup V_2|\ge4$ and in $G_v\setminus\{u,v\}$ all vertices of $V_0\cup V_1$ are adjacent to all vertices of $V_2$, thus $G_v\setminus\{u,v\}$ contains a triangle. Hence $\reg(J_{G_v\setminus\{u,v\}})\le|V(G_v\setminus\{u,v\})|-1=n-3$. Note that in $(G_v\setminus v)_u$, $T$ is a clique. As before, we have $\reg(J_{(G_v\setminus v)_u})\le(n-1)+2-|T|=3\le n-3$. Moreover, $(G_v\setminus v)_u\setminus u$ is a complete graph. Consequently, by Theorem \ref{Thm:regGbounds}(i) we have $\reg(J_{(G_v\setminus v)_u\setminus u})+1=3\le n-3$. Thus,
		\begin{equation}\label{eq:bound|V_0|1p12}
		\reg(J_{G_v\setminus v})\le n-3.
		\end{equation}
		Finally, combining (\ref{eq:bound|V_0|1p1}) with (\ref{eq:bound|V_0|1p6}), (\ref{eq:bound|V_0|1p11}) and (\ref{eq:bound|V_0|1p12}) we obtain that $\reg(J_G)\le n-2$, as desired. The proof is complete.
	\end{proof}
	\begin{Corollary}\label{Cor:pd=2n-6}
		Let $n\ge5$ and $3\le r\le n-2$ be positive integers. Then, there exists a graph $G$ on $n$ non-isolated vertices such that
		$$\pd(J_G)=2n-6\ \ \ \text{and}\ \ \ \reg(J_G)=r.$$
	\end{Corollary}
	\begin{proof}
		Let $3\le r\le n-3$, and set $\widetilde{G}=P_{r}\sqcup(n-r-3)K_1$ and $G=\widetilde{G}*3K_1$. Then, by Theorem \ref{Thm:pd=2n-6}, $\pd(J_G)=2n-6$, and by formula (\ref{eq:RegJoinGraphs}), $\reg(J_G)=\reg(J_{\widetilde{G}})=\reg(J_{P_r})=r$. For $n=5$, we have $r=3$ and we can apply Corollary \ref{Cor:reg=3}.
		
		Let now $n\ge6$ and $r=n-2$. Let $G$ be the graph on the vertex set $V(G)=\{u,v,1,\dots,n-2\}$ and with edge set $E(G)$ equal to
		$$
		\big\{\{i,i+1\},\{i,u\}:i=1,\dots,n-3\big\}\cup\big\{\{i,v\}:i=1,\dots,n-4\big\}\cup\{\{n-2,v\}\}.
		$$
		Then, setting $T=\{1,\dots,n-2\}$, we have that $G\in\mathcal{G}_T$. To see why this is true, using the same notation as in Definition \ref{Def:defG_T}, just take $V_0=\{1,\dots,n-4\}$, $V_1=\{n-3\}$ and $V_2=\{n-2\}$. Therefore, by Theorem \ref{Thm:pd=2n-6}, $\pd(J_G)=2n-6$. Note that the induced subgraph $P=G_{\{1,\dots,n-2\}}$ is a path on $n-2$ vertices. Hence, $\reg(J_G)\ge\reg(J_P)=n-2$. On the other hand, by Proposition \ref{Prop:pd=2n-6}, $\reg(J_G)\le n-2$. Consequently, $\reg(J_G)=n-2$ and $G$ is the graph we are looking for. 
	\end{proof}
	
	\subsection{Graphs with regularity 3} We quote the following result \cite[Theorem 3.2]{MK18}.
	\begin{Theorem}\label{Thm:reg=3}
		Let $G$ be a non-complete graph on $n$ non-isolated vertices. Then $\reg(J_G)=3$ if and only if one of the following conditions holds:
		\begin{enumerate}
			\item[\textup{(i)}] $G=K_r\sqcup K_s$ with $r,s\ge2$ and $r+s=n$, or
			\item[\textup{(ii)}] $G=G_1*G_2$, where $G_i$ is a graph on $n_i<n$ vertices such that $n_1+n_2=n$ and $\reg(J_{G_i})\le3$, for $i=1,2$.
		\end{enumerate}
	\end{Theorem}
	
	\begin{Remark}\label{Rem:reg=3,GConnected}
		\rm If $G$ is a graph on $n\ge4$ non-isolated vertices with $\reg(J_G)=3$ and $\pd(J_G)\ge n-2$, then $G$ must be connected. Suppose by contradiction that there exists a disconnected graph $G=G_1\sqcup\dots\sqcup G_c$ with regularity $r=3$ and projective dimension $\pd(J_G)\ge n-2$. Since $r=3\le\lfloor\frac{n}{2}\rfloor+1$, using the calculation (\ref{eq:laterused}) we obtain $c\le r-1=2$. Thus $c=2$ and by formula (\ref{eq:regJGc}) we must have $\reg(J_{G_1})+\reg(J_{G_2})-1=3$. This formula holds if and only if $\reg(J_{G_1})=\reg(J_{G_2})=2$. Thus $G=K_r\sqcup K_s$ as in Theorem \ref{Thm:reg=3}(i). But then formula (\ref{eq:pdJGc}) and Theorem \ref{Thm:regGbounds}(i) yield $\pd(J_G)=n-3$, a contradiction.
	\end{Remark}
	
	For the proof of the next result, we need the following lemma which is an immediate consequence of \cite[Theorems 3.4 and 3.9]{KS20}.
	\begin{Lemma}\label{Lemma:KumarSarkar}
		Let $\widetilde{G}$ be a graph on $n-1$ vertices, and set $G=\widetilde{G}*K_1$. Then,
		$$
		\pd(J_G)=\begin{cases}
			\pd(J_{\widetilde{G}})+2,&\text{if}\ \widetilde{G}\ \text{is connected},\\
			\max\{\pd(J_{\widetilde{G}})+2,n-3\},&\text{if}\ \widetilde{G}\ \text{is disconnected}.
		\end{cases}
		$$
	\end{Lemma}
	\begin{proof}
		If $\widetilde{G}$ is connected, by \cite[Theorem 3.4]{KS20}, $\depth_S(S/J_G)=\depth_{\widetilde{S}}(\widetilde{S}/J_{\widetilde{G}})$, where $\widetilde{S}=K[x_v,y_v:v\in V(\widetilde{G})]$. Using the Auslander--Buchsbaum formula, we obtain $2n-\pd(J_G)=2(n-1)-\pd(J_{\widetilde{G}})$, and consequently, $$\pd(J_G)=\pd(J_{\widetilde{G}})+2.$$
		
		Suppose now that $\widetilde{G}$ is disconnected. By \cite[Theorem 3.9]{KS20},
		$$
		\depth_S(S/J_G)=\min\{\depth_{\widetilde{S}}(\widetilde{S}/J_{\widetilde{G}}),n+2\}.
		$$
		Using the Auslander--Buchsbaum formula we obtain
		$$
		2n-\pd(J_G)-1=\min\{2(n-1)-\pd(J_{\widetilde{G}})-1,n+2\}.
		$$
		Therefore,
		\begin{align*}
			\pd(J_G)\ &=\ 2n-1-\min\{2(n-1)-1-\pd(J_{\widetilde{G}}),n+2\}\\
			&=\ \max\{\pd(J_{\widetilde{G}})+2,n-3\},
		\end{align*}
		as desired.
	\end{proof}

	\begin{Corollary}\label{Cor:reg=3}
		Let $n\ge4$ be an integer. Then, for all $n-3\le p\le 2n-5$, there exists a graph $G$ on $n$ non-isolated vertices such that
		$$
		\pd(J_G)=p\ \ \ \text{and}\ \ \ \reg(J_G)=3.
		$$
	\end{Corollary}
	\begin{proof}
		We proceed by induction on $n\ge4$. Suppose $n=4$. Then, the binomial edge ideals of the graphs $2K_2$, $(P_2\sqcup K_1)*K_1$ and $P_2*2K_1$ have regularity $3$ and projective dimension, respectively, 1, 2 and 3.
		
		Let $n>4$. If $p=2n-5$ or $p=2n-6$, then a graph $G$ on $n$ non-isolated vertices with $\pd(J_G)=p$ and $\reg(J_G)=3$ exists by Corollaries \ref{Cor:pd=2n-5}, \ref{Cor:pd=2n-6}. So, we can assume $n-3\le p\le 2n-7$. If $p=n-3$, then the binomial edge ideal of $K_r\sqcup K_s$ with $r,s\ge2$ and $r+s=n$ has projective dimension $n-3$ and regularity 3, by Theorem \ref{Thm:reg=3}(i). Now, let $n-2\le p\le 2n-7$. Then $(n-1)-3\le p-2\le 2(n-1)-7<2(n-1)-5$. Hence, by induction there exists a graph $\widetilde{G}$ on $n-1$ vertices with $\pd(J_{\widetilde{G}})=p-2$ and $\reg(J_{\widetilde{G}})=3$. Let $G=\widetilde{G}*K_1$. If $p-2=(n-1)-3$, then $\widetilde{G}$ is disconnected (Theorem \ref{Thm:pdGconnect}). By Lemma \ref{Lemma:KumarSarkar},
		\begin{align*}
		\pd(J_G)&=\max\{\pd(J_{\widetilde{G}})+2,n-3\}\\&=\max\{p,n-3\}=\max\{n-2,n-3\}=n-2.
		\end{align*}
		Otherwise, if $p-2\ge(n-1)-2$, then $\widetilde{G}$ is connected by Remark \ref{Rem:reg=3,GConnected}. Then, by Lemma \ref{Lemma:KumarSarkar}, $\pd(J_{G})=\pd(J_{\widetilde{G}})+2=p$, as desired.
	\end{proof}
	
	\subsection{Graphs with regularity $n-2$}
	In the next result, we consider graphs $G$ on $n$ non-isolated vertices having regularity $r=\reg(J_G)=n-2$. If $n=5$, then $r=3$ and we can apply Corollary \ref{Cor:reg=3}. Therefore, we assume $n\ge6$. In this case $r=n-2\ge\lceil\frac{n}{2}\rceil+1$. Thus, by Proposition \ref{Prop:refinedPd2}, the projective dimension for such graphs varies as follows: $r-2=n-4\le\pd(J_G)\le 2n-5$.\smallskip
	
	For the next result, we need the concept of \textit{decomposable graph} introduced by Rauf and Rinaldo in \cite{RR14}. A graph $G$ is called \textit{decomposable} if there exist two graphs $G_1$ and $G_2$ such that $V(G)=V(G_1)\cup V(G_2)$, $V(G_1)\cap V(G_2)=\{v\}$ where $v$ is a simplicial vertex for both $G_1$ and $G_2$, and such that $E(G)=E(G_1)\cup E(G_2)$. In such case, we write $G=G_1\cup_v G_2$ and say that $G$ \textit{is obtained by gluing $G_1$ and $G_2$ along $v$}. By \cite[Proposition 1.3]{HR18} it follows that
	\begin{align*}
	\pd(J_G)\ &=\ \pd(J_{G_1})+\pd(J_{G_2})+1,\\
	\reg(J_G)\ &=\ \reg(J_{G_1})+\reg(J_{G_2})-1.
	\end{align*}
	
	\begin{Lemma}\label{Lemma:CompletePdBound}
		Let $G$ be the graph with vertex set $V(G)=[n]$ and edge set
		$$
		E(G)=\big\{\{i,j\}:1\le i<j\le n-1\big\}\cup\big\{\{m,n\},\{m+1,n\},\dots,\{n-1,n\}\big\},
		$$
		for some $m\in[n-1]$. Then $\pd(J_G)\le 2n-m-3$ and $\reg(J_G)\le3$.
	\end{Lemma}
	\begin{proof}
		For the regularity, note that in $G$ the set $W=[n-1]$ is a clique. Therefore, $\reg(J_G)\le n+2-|W|=3$. For the projective dimension, we proceed by induction on $n\ge2$. For $n=2$, $m=1$, $G$ is the path on 2 vertices and the statement is trivial.
			
		Let $n>2$ and $m = n-1$. Then $G$ is decomposable as $G=G_{[n-1]}\cup_{n-1}G_{\{n-1,n\}}$. Note that $G_{[n-1]}$ is a complete graph and $J_{G_{\{n-1,n\}}}$ is a principal ideal. Thus,
		\begin{align*}
		\pd(J_G)&=\pd(J_{{G_{[n-1]}}})+\pd(J_{G_{\{n-1,n\}}})+1\\&=(n-1)-2+0+1=n-2.
		\end{align*}
		Since $m=n-1$, then $2n-m-3=n-2$. Hence this case is verified.
		
		Suppose now $m<n-1$. Then, $m$ is an internal vertex of $G$, because it belongs to two different maximal cliques, namely $[n-1]$ and $\{m,m+1,\dots,n\}$. Applying Ohtani lemma to $m$, by the short exact sequence (\ref{eq:ShortExactSeqOhtani}) we obtain
		\begin{align*}
		     &\pd(J_G)\le\\
		     &\ \ \ \le\max\{\pd(x_m,y_m,J_{G\setminus m}),\pd(J_{G_m}),\pd(x_m,y_m,J_{G_m\setminus m})-1\}.
		\end{align*}
		The graph $G\setminus m$ on $n-1$ vertices is of the type described in the statement.
		Indeed, we can relabel the vertices of the graph $G\setminus m$ in the following way: the labels of
		vertices from 1 to $m-1$ remain unchanged, while the labels of the vertices from $m+1$
		to $n$ are each decreased by 1. Thus, by the inductive hypothesis we have $\pd(J_{G\setminus m})\le 2(n-1)-m-3=2n-m-5$, and so
		$$
		\pd(x_m,y_m,J_{G\setminus m})=\pd(J_{G\setminus m})+2\le 2n-m-3.
		$$
		For the other two inequalities, note that $G_m$ and $G_m\setminus m$ are complete graphs on $n$ and $n-1$ vertices, respectively. Hence, Theorem \ref{Thm:regGbounds}(i) gives
		\begin{align*}
		\pd(J_{G_m})&=n-2,\\
		\pd(x_m,y_m,J_{G_m\setminus m})-1&=\pd(J_{G_m\setminus m})+2-1=n-3+1=n-2.
		\end{align*}
		But $n-2\le 2n-m-3$ because $m\le n-1$, by construction. Finally, all the inequalities obtained show that $\pd(J_G)\le 2n-m-3$, as desired.
	\end{proof}
	
	We need the following technique. Let $e=\{u,v\}\in E(G)$. By $G\setminus e$ we denote the graph with $V(G\setminus e)=V(G)$ and $E(G\setminus e)=E(G)\setminus\{e\}$. By $G_e$ we denote the graph with $V(G_e)=V(G)$ and $E(G_e)=E(G)\cup\{\{w_1,w_2\}:w_1,w_2\in N_G(u)\ \text{or}\ w_1,w_2\in N_G(v)\}$. Set $f_e=x_uy_v-x_vy_u$. Then, we have a short exact sequence
	\begin{equation}\label{eq:ShortExactSeqRemoveEdge}
	0\rightarrow\frac{S}{(J_{G\setminus e}:f_e)}(-2)\longrightarrow\frac{S}{J_{G\setminus e}}\longrightarrow\frac{S}{J_G}\rightarrow 0.
	\end{equation}
	
	By \cite[Theorem 3.7]{MS14}, we have
	\begin{equation}\label{eq:GminusEdge}
	J_{G\setminus e}:f_e=J_{(G\setminus e)_e}+I_G,
	\end{equation}
	where $$I_G=(g_{P,t}:P:u,u_1,\dots,u_s,v\ \textit{is a path between}\ u\ \textit{and}\ v\ \textit{in}\ G\ \textit{and}\ 0\le t\le s),$$ with $g_{P,0}=x_{u_1}\cdots x_{u_s}$ and $g_{P,t}=y_{u_1}\cdots y_{u_t}x_{u_{t+1}}\cdots x_{u_s}$ for every $1\le t\le s$.
	
	\begin{Proposition}\label{Prop:reg=n-2}
		Let $n\ge6$ be a positive integer. Then, for all $n-4\le p\le 2n-5$, there exists a graph $G$ on $n$ non-isolated vertices, such that
		$$
		\pd(J_G)=p\ \ \ \text{and}\ \ \ \reg(J_G)=n-2.
		$$
	\end{Proposition}
	\begin{proof}
		If $p=n-4$, then $G=P_2\sqcup P_2\sqcup P_{n-4}$ has $\pd(J_G)=n-4$ and $\reg(J_G)=n-2$. If $p=n-3$, then $G=K_3\sqcup P_{n-3}$ has $\pd(J_G)=n-3$ and $\reg(J_G)=n-2$. If $p=n-2$, then consider the graph depicted below
		
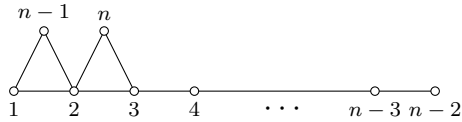
\begin{center}
			\begin{tikzpicture}[scale=0.8]
			\draw[-] (1,0) -- (8,0);
			\draw[-] (1,0) -- (1.5,1) -- (2,0) -- (2.5,1) -- (3,0);
			\filldraw[black, fill=white] (1,0) circle (2pt) node[below]{{\tiny1}};
			\filldraw[black, fill=white] (2,0) circle (2pt) node[below]{{\tiny2}};
			\filldraw[black, fill=white] (3,0) circle (2pt) node[below]{{\tiny3}};
			\filldraw[black, fill=white] (4,0) circle (2pt) node[below]{{\tiny4}};
			\filldraw[black, fill=white] (5.5,0) node[below]{$\cdots$};
			\filldraw[black, fill=white] (7,0) circle (2pt) node[below]{{\tiny$n-3$}};
			\filldraw[black, fill=white] (8,0) circle (2pt) node[below]{{\tiny$n-2$}};
			\filldraw[black, fill=white] (1.5,1) circle (2pt) node[above]{{\tiny$n-1$}};
			\filldraw[black, fill=white] (2.5,1) circle (2pt) node[above]{{\tiny$n$}};
			\end{tikzpicture}
		\end{center}\vspace*{-0.2cm}
		It is clear that $G$ is decomposable as $(G_{\{1,2,n-1\}}\cup_2 G_{\{2,3,n\}})\cup_3 G_{\{3,4,\dots,n-2\}}$. We set $G_1=G_{\{1,2,n-1\}}$, $G_2=G_{\{2,3,n\}}$ and $G_3=G_{\{3,4,\dots,n-2\}}$. $G_1$ and $G_2$ are complete graphs with three vertices each, while $G_3$ is a path on $n-4$ vertices. Therefore,
		\begin{align*}
		    \pd(J_G)&=\sum_{i=1}^3\pd(J_{G_i})+2=1+1+n-6+2=n-2,\\
		    \reg(J_G)&=\sum_{i=1}^3\reg(J_{G_i})-2=2+2+n-4-2=n-2.
		\end{align*}
		
		It remains to consider the case $n-1\le p\le 2n-5$. For this purpose, let $m\in\{1,\dots,n-3\}$ and consider the graph $G$ depicted below.
		
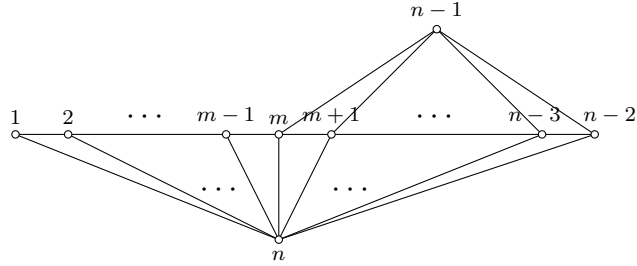
\begin{center}
			\begin{tikzpicture}[scale=0.7]
			\draw[-] (12,0) -- (7,-2);
			\draw[-] (13,0) -- (7,-2);
			\draw[-] (7,0) -- (10,2);
			\draw[-] (8,0) -- (10,2);
			\draw[-] (12,0) -- (10,2);
			\draw[-] (13,0) -- (10,2);
			\draw[-] (6,0) -- (7,-2);
			\draw[-] (7,0) -- (7,-2);
			\draw[-] (8,0) -- (7,-2);
			\draw[-] (2,0) -- (13,0);
			\draw[-] (2,0) -- (7,-2);
			\draw[-] (3,0) -- (7,-2);
			\filldraw[black, fill=white] (2,0) circle (2pt) node[above]{{\tiny1}};
			\filldraw[black, fill=white] (3,0) circle (2pt) node[above]{{\tiny2}};
			\filldraw[black, fill=white] (4.5,0) node[above]{$\cdots$};
			\filldraw[black, fill=white] (6,0) circle (2pt) node[above]{{\tiny$m-1$}};
			\filldraw[black, fill=white] (7,0) circle (2pt) node[above]{{\tiny$m$}};
			\filldraw[black, fill=white] (8,0) circle (2pt) node[above]{{\tiny$m+1$}};
			\filldraw[black, fill=white] (10,0) node[above]{$\cdots$};
			\filldraw[black, fill=white] (12,0) circle (2pt) node[above,xshift=-0.1cm]{{\tiny$n-3$}};
			\filldraw[black, fill=white] (13,0) circle (2pt) node[above,xshift=0.2cm]{{\tiny$n-2$}};
            \filldraw[black, fill=white] (7,-2) circle (2pt) node[below]{{\tiny$n$}};
            \filldraw[black, fill=white] (10,2) circle (2pt) node[above]{{\tiny$n-1$}};
			\filldraw[black, fill=white] (5.9,-1.4) node[above]{$\cdots$};
			\filldraw[black, fill=white] (8.4,-1.4) node[above]{$\cdots$};
			\end{tikzpicture}
		\end{center}\vspace*{-0.2cm}
	    That is, $V(G)=[n]$ and
	    \begin{align*}
	     E(G)\ =&\ \big\{\{i,i+1\}:i=1,\dots,n-3\big\}\\\cup&\ \big\{\{i,n-1\}:i=m,\dots,n-2\big\}\\\cup&\ \big\{\{i,n\}:i=1,\dots,n-2\big\}.
	    \end{align*}
		We claim that $\pd(J_G)=2n-m-4$ and $\reg(J_G)=n-2$. Since $$\{2n-m-4:m=1,\dots,n-3\}=\{n-1,n,\dots,2n-5\},$$ the claim will conclude the proof.
		
		Firstly, we note that if $m=1$, then $G=G_{[n-2]}*G_{\{n-1,n\}}$. Since $G_{[n-2]}$ is a path on $n-2$ vertices and $G_{\{n-1,n\}}$ consists of two isolated vertices, by Theorem \ref{Thm:pd=2n-5} and formula (\ref{eq:RegJoinGraphs}) we obtain
		$$
		\pd(J_G)=2n-5,\ \ \ \reg(J_G)=n-2.
		$$
		
		If $m=2$, then $G$ is a $D_5$-type graph. Indeed, $G\in\mathcal{G}_{[n-2]}$. To see why this is true, using the notation of Definition \ref{Def:defG_T}, it is enough to set $u=n-1$, $v=1$, $V_0=\{2\}$, $V_1=\{3,\dots,n-2\}$ and $V_2=\{n\}$. Hence $\pd(J_G)=2n-6$. Furthermore, $\reg(J_G)\ge\reg(J_{G_{[n-2]}})=n-2$, but also $\reg(J_{G})\le n-2$ by Proposition \ref{Prop:pd=2n-6}.
		
		Thus our claim holds for $m=1,2$. Now, we proceed by induction on $n\ge5$. If $n=5$, then $m\in\{1,2\}$ and there is nothing to prove. Suppose $n\ge6$ and let $m\in\{3,\dots,n-3\}$. Set $e=\{1,n\}$ and let $f_e=x_1y_n-x_ny_1$. Then, by the short exact sequence (\ref{eq:ShortExactSeqRemoveEdge}), see also \cite[Proposition 2.1(a)]{KM16}, we have
		\begin{equation}\label{eq:pdboundn-21}
		\reg(J_G)\le\max\{\reg(J_{G\setminus e}),\reg(J_{G\setminus e}:f_e)+1\}.
		\end{equation}
		Note that $G\setminus e$ is decomposable as $G\setminus e=G_{\{1,2\}}\cup_2 G_{\{2,\dots,n\}}$. Using the induction on $G_{\{2,\dots,n\}}$ we have
		\begin{align}
		\nonumber\pd(J_{G\setminus e})&=\pd(J_{G_{\{1,2\}}})+\pd(J_{G_{\{2,\dots,n\}}})+1\\
		\label{eq:pdboundn-22}&=2(n-1)-(m-1)-4+1\\
		\nonumber&=2n-m-4,\\[0.2em]
		\label{eq:pdboundn-23}\reg(J_{G\setminus e})&=\reg(J_{G_{\{1,2\}}})+\reg(J_{G_{\{2,\dots,n\}}})-1=2+n-3-1=n-2.
		\end{align}
		
		By equation (\ref{eq:GminusEdge}), $J_{G\setminus e}:f_e=J_{(G\setminus e)_e}+I_G$. Note that any path in $G$ from $1$ to $n$ different from the path $1,n$ must contain the vertex 2. Hence $I_G=(x_2,y_2)$. Consequently, $J_{G\setminus e}:f_e=(J_{\widetilde{G}},x_2,y_2)$, where $\widetilde{G}$ is the graph $(G\setminus e)_e\setminus\{2\}$ with $V(\widetilde{G})=\{3,\dots,n\}$ and
		\begin{align*}
		E(\widetilde{G})\ =&\ \big\{\{i,j\}:3\le i<j\le n-2\big\}\\
		\cup&\ \big\{\{i,n\}:i=3,\dots,n-2\big\}\\
		\cup&\ \big\{\{j,n-1\}:j=m,\dots,n-2\big\}.
		\end{align*}
		Therefore, using Lemma \ref{Lemma:CompletePdBound} applied on $\widetilde{G}$ we get
		\begin{align}
		\nonumber\pd(J_{G\setminus e}:f_e)&=\pd(J_{\widetilde{G}},x_2,y_2)=\pd(J_{\widetilde{G}})+2\\
		\label{eq:pdboundn-24}&\le2(n-2)-m-3+2\\
		\nonumber&=2n-m-5,\\[0.2em]
		\label{eq:pdboundn-25}\reg(J_{G\setminus e}:f_e)+1&=\reg(J_{\widetilde{G}})+1\le 4\le n-2,
		\end{align}
		as $n\ge 6$. By (\ref{eq:pdboundn-22}) and (\ref{eq:pdboundn-24}) we have $\pd(J_{G\setminus e}:f_e)<\pd(J_{G\setminus e})$. Thus, using the short exact sequence (\ref{eq:ShortExactSeqRemoveEdge}) we obtain
		$$
		\pd(J_G)=\pd(J_{G\setminus e})=2n-m-4.
		$$
		Whereas, combining (\ref{eq:pdboundn-21}) with (\ref{eq:pdboundn-23}) and (\ref{eq:pdboundn-25}) we obtain $\reg(J_G)\le n-2$. Since also $\reg(J_G)\ge\reg(J_{G_{[n-2]}})=n-2$, as $G_{[n-2]}$ is a path on $n-2$ vertices, we obtain the equality. The inductive proof is complete.
	\end{proof}
	
	\begin{Remark}\label{Rem:Gconnectr=n-2}
		\rm Note that the graphs constructed in the previous proposition, for $p\ge n-2$, are all connected.
	\end{Remark}
	
	\subsection{The size of betti tables of binomial edge ideals of small graphs.}
	
	By putting together all results in this section we can determine the set
	$$
	\textup{pdreg}(n)\ =\ \big\{(\pd(J_G),\reg(J_G)):G\in\textup{Graphs}(n)\big\},
	$$
	for small values of $n$, where $\textup{Graphs}(n)$ denotes the class of all finite simple graphs on $n$ non-isolated vertices.
	
	\begin{Example}\label{Ex:InitialCasesPdReg}
		\rm We determine the set $\textup{pdreg}(n)$ for $n=3,4,5$ and $6$.
		\begin{enumerate}
			\item[\footnotesize{$({\bf n=3})$}] We have $\textup{pdreg}(3)=\{(1,2),(1,3)\}$. In the following list we display all the graphs $G$ on three non-isolated vertices and below each of them the pair $(\pd(J_G),\reg(J_G))$.
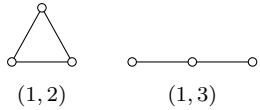
\bigskip
				\begin{center}
					\begin{tikzpicture}[scale=0.8]
					\draw[-] (1,0) -- (2,0) -- (1.5,0.9) -- (1,0);
					\draw[-] (3,0) -- (4,0) -- (5,0);
					\filldraw[black, fill=white] (1,0) circle (2pt);
					\filldraw[black, fill=white] (2,0) circle (2pt);
					\filldraw[black, fill=white] (1.5,0.9) circle (2pt);
					\filldraw[black, fill=white] (3,0) circle (2pt);
					\filldraw[black, fill=white] (4,0) circle (2pt);
					\filldraw[black, fill=white] (5,0) circle (2pt);
					\filldraw[black, fill=white] (1.5,-0.2) node[below]{{\tiny$(1,2)$}};
					\filldraw[black, fill=white] (4,-0.2) node[below]{{\tiny$(1,3)$}};
					\end{tikzpicture}
				\end{center}\medskip
			\item[\footnotesize{$({\bf n=4})$}] We have $\textup{pdreg}(4)=\{(2,2),(1,3),(2,3),(3,3),(2,4)\}$. The following is a list of graphs on four non-isolated vertices that gives such pairs.\bigskip
			
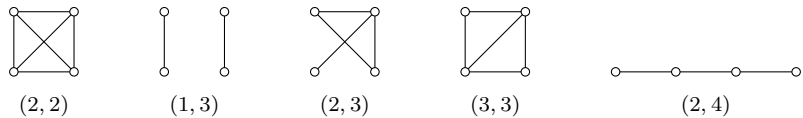
\begin{center}
				\begin{tikzpicture}[scale=0.8]
				\draw[-] (1,0) -- (2,0) -- (2,1) -- (1,1) -- (1,0) -- (2,1);
				\draw[-] (2,0) -- (1,1);
				\draw[-] (3.5,0) -- (3.5,1);
				\draw[-] (4.5,0) -- (4.5,1);
				\draw[-] (6,0) -- (7,1);
				\draw[-] (6,1) -- (7,1) -- (7,0) -- (6,1);
				\draw[-] (8.5,0) -- (8.5,1) -- (9.5,1) -- (9.5,0) -- (8.5,0) -- (9.5,1);
				\draw[-] (11,0) -- (12,0) -- (13,0) -- (14,0);	
				\filldraw[black, fill=white] (1,0) circle (2pt);
				\filldraw[black, fill=white] (2,0) circle (2pt);
				\filldraw[black, fill=white] (1,1) circle (2pt);
				\filldraw[black, fill=white] (2,1) circle (2pt);
				\filldraw[black, fill=white] (1.5,-0.2) node[below]{{\tiny$(2,2)$}};
				\filldraw[black, fill=white] (3.5,0) circle (2pt);
                \filldraw[black, fill=white] (4.5,0) circle (2pt);
                \filldraw[black, fill=white] (3.5,1) circle (2pt);
                \filldraw[black, fill=white] (4.5,1) circle (2pt);
                \filldraw[black, fill=white] (4,-0.2) node[below]{{\tiny$(1,3)$}};
                \filldraw[black, fill=white] (6,0) circle (2pt);
                \filldraw[black, fill=white] (7,0) circle (2pt);
                \filldraw[black, fill=white] (6,1) circle (2pt);
                \filldraw[black, fill=white] (7,1) circle (2pt);
                \filldraw[black, fill=white] (6.5,-0.2) node[below]{{\tiny$(2,3)$}};
                \filldraw[black, fill=white] (8.5,0) circle (2pt);
                \filldraw[black, fill=white] (9.5,0) circle (2pt);
                \filldraw[black, fill=white] (8.5,1) circle (2pt);
                \filldraw[black, fill=white] (9.5,1) circle (2pt);
                \filldraw[black, fill=white] (9,-0.2) node[below]{{\tiny$(3,3)$}};
                \filldraw[black, fill=white] (11,0) circle (2pt);
                \filldraw[black, fill=white] (12,0) circle (2pt);
                \filldraw[black, fill=white] (13,0) circle (2pt);
                \filldraw[black, fill=white] (14,0) circle (2pt);
                \filldraw[black, fill=white] (12.5,-0.2) node[below]{{\tiny$(2,4)$}};						
				\end{tikzpicture}
			\end{center}\medskip
		    Note that the second, third and fourth graph are, respectively, $K_2\sqcup K_2$, $(P_2\sqcup K_1)*K_1$, $P_2*2K_1$. These graphs have regularity $3$, by Theorem \ref{Thm:reg=3}.\medskip
		    \item[\footnotesize{$({\bf n=5})$}] It is $\textup{pdreg}(5)=\{(3,2),(2,3),(3,3),(4,3),(5,3),(2,4),(3,4),(4,4),(3,5)\}$. Furthermore, a list of graphs giving such pairs is given below.\bigskip
			
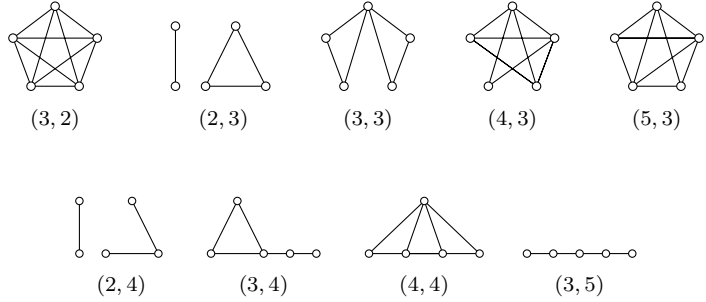
\begin{center}
				\begin{tikzpicture}[scale=0.8]
				\draw[-] (0.8,0.8) -- (1.1,0) -- (1.9,0) -- (2.2,0.8) -- (1.5,1.33) -- (1.1,0) -- (2.2,0.8) -- (0.8,0.8) -- (1.9,0) -- (1.5,1.33) -- (0.8,0.8);
				\draw[-] (3.5,0) -- (3.5,1);
				\draw[-] (4,0) -- (5,0) -- (4.5,1) -- (4,0);
				\draw[-] (6,0.8) -- (6.7,1.33) -- (6.3,0) -- (6,0.8);
				\draw[-] (7.1,0) -- (6.7,1.33) -- (7.4,0.8) -- (7.1,0);
				\draw[-] (9.5,0) -- (8.4,0.8) -- (9.5,0) -- (9.8,0.8) -- (9.1,1.33) -- (8.7,0) -- (9.8,0.8) -- (9.5,0) -- (9.1,1.33) -- (8.4,0.8) -- (9.8,0.8);
				\draw[-] (11.9,0) -- (11.1,0) -- (10.8,0.8) -- (12.2,0.8) -- (11.5,1.33) -- (11.1,0) -- (12.2,0.8) -- (11.9,0) -- (11.5,1.33) -- (10.8,0.8) -- (12.2,0.8);
				\filldraw[black, fill=white] (0.8,0.8) circle (2pt);
				\filldraw[black, fill=white] (1.1,0) circle (2pt);
				\filldraw[black, fill=white] (1.9,0) circle (2pt);
				\filldraw[black, fill=white] (2.2,0.8) circle (2pt);
				\filldraw[black, fill=white] (1.5,1.33) circle (2pt);
				\filldraw[black, fill=white] (1.5,-0.2) node[below]{{\tiny$(3,2)$}};
				\filldraw[black, fill=white] (3.5,0) circle (2pt);
				\filldraw[black, fill=white] (3.5,1) circle (2pt);
				\filldraw[black, fill=white] (4,0) circle (2pt);
				\filldraw[black, fill=white] (5,0) circle (2pt);
				\filldraw[black, fill=white] (4.5,1) circle (2pt);
				\filldraw[black, fill=white] (4.3,-0.2) node[below]{{\tiny$(2,3)$}};
				\filldraw[black, fill=white] (6,0.8) circle (2pt);
				\filldraw[black, fill=white] (6.3,0) circle (2pt);
				\filldraw[black, fill=white] (7.1,0) circle (2pt);
				\filldraw[black, fill=white] (7.4,0.8) circle (2pt);
				\filldraw[black, fill=white] (6.7,1.33) circle (2pt);
				\filldraw[black, fill=white] (6.7,-0.2) node[below]{{\tiny$(3,3)$}};
				\filldraw[black, fill=white] (8.4,0.8) circle (2pt);
				\filldraw[black, fill=white] (8.7,0) circle (2pt);
				\filldraw[black, fill=white] (9.5,0) circle (2pt);
				\filldraw[black, fill=white] (9.8,0.8) circle (2pt);
				\filldraw[black, fill=white] (9.1,1.33) circle (2pt);
				\filldraw[black, fill=white] (9.1,-0.2) node[below]{{\tiny$(4,3)$}};
				\filldraw[black, fill=white] (10.8,0.8) circle (2pt);
				\filldraw[black, fill=white] (11.1,0) circle (2pt);
				\filldraw[black, fill=white] (11.9,0) circle (2pt);
				\filldraw[black, fill=white] (12.2,0.8) circle (2pt);
				\filldraw[black, fill=white] (11.5,1.33) circle (2pt);
				\filldraw[black, fill=white] (11.5,-0.2) node[below]{{\tiny$(5,3)$}};
				\end{tikzpicture}
			\end{center}\medskip
		    \begin{center}
		    	\begin{tikzpicture}[scale=0.7]
		    	\draw[-] (3.5,0) -- (3.5,1);
		    	\draw[-] (4,0) -- (5,0) -- (4.5,1);
		    	\draw[-] (7,0) -- (6.5,1) -- (6,0) -- (8,0);
		    	\draw[-] (9,0) -- (10.05,1) -- (9.7,0) -- (11.1,0) -- (10.05,1) -- (10.4,0) -- (9,0);
		    	\draw[-] (12,0) -- (14,0);
		    	\filldraw[black, fill=white] (3.5,0) circle (2pt);
		    	\filldraw[black, fill=white] (3.5,1) circle (2pt);
		    	\filldraw[black, fill=white] (4,0) circle (2pt);
		    	\filldraw[black, fill=white] (5,0) circle (2pt);
		    	\filldraw[black, fill=white] (4.5,1) circle (2pt);
		    	\filldraw[black, fill=white] (4.3,-0.2) node[below]{{\tiny$(2,4)$}};
		    	\filldraw[black, fill=white] (6,0) circle (2pt);
		    	\filldraw[black, fill=white] (6.5,1) circle (2pt);
		    	\filldraw[black, fill=white] (7,0) circle (2pt);
		    	\filldraw[black, fill=white] (7.5,0) circle (2pt);
		    	\filldraw[black, fill=white] (8,0) circle (2pt);
		    	\filldraw[black, fill=white] (7,-0.2) node[below]{{\tiny$(3,4)$}};
		    	\filldraw[black, fill=white] (9,0) circle (2pt);
		    	\filldraw[black, fill=white] (10.05,1) circle (2pt);
		    	\filldraw[black, fill=white] (9.7,0) circle (2pt);
		    	\filldraw[black, fill=white] (10.4,0) circle (2pt);
		    	\filldraw[black, fill=white] (11.1,0) circle (2pt);
		    	\filldraw[black, fill=white] (10.05,-0.2) node[below]{{\tiny$(4,4)$}};
		    	\filldraw[black, fill=white] (12,0) circle (2pt);
		    	\filldraw[black, fill=white] (12.5,0) circle (2pt);
		    	\filldraw[black, fill=white] (13,0) circle (2pt);
		    	\filldraw[black, fill=white] (13.5,0) circle (2pt);
		    	\filldraw[black, fill=white] (14,0) circle (2pt);
		    	\filldraw[black, fill=white] (13,-0.2) node[below]{{\tiny$(3,5)$}};
		    	\end{tikzpicture}
		    \end{center}\medskip
	    	Note that each graph $G$ displayed, with $\pd(J_G)\ge n-2=3$, is connected. The graphs with regularity 3 are constructed as in Corollary \ref{Cor:reg=3}. They are, in the given order: $K_2\sqcup K_3$, $(K_2\sqcup K_2)*K_1$, $((P_2\sqcup K_1)*K_1)*K_1$, and $(P_2*2K_1)*K_1$. Moreover, since $n=5$, by Corollary \ref{Cor:pd=2n-5}, if $\reg(J_G)=n-1=4$ then $\pd(J_G)\le 2n-6=4$.
	        \item[\footnotesize{$({\bf n=6})$}] In the following, we list all pairs of the set $\textup{pdreg}(6)$, and for each pair $(p,r)$ in the set a graph $G$ with $(\pd(J_G),\reg(J_G))=(p,r)$.
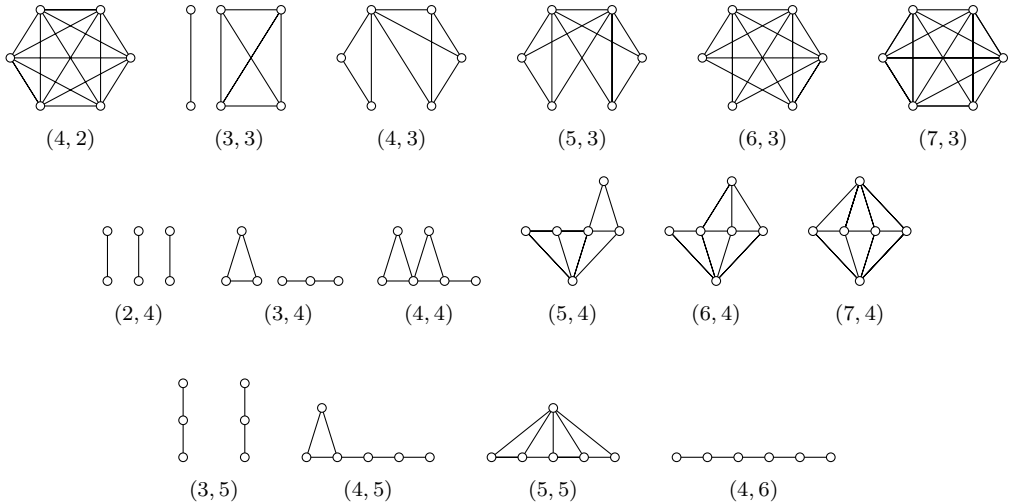
\bigskip
	        \begin{center}
	        	\begin{tikzpicture}[scale=0.8]
	        	\draw[-] (1,0) -- (1,1.6) -- (1.5,0.8) -- (0,1.6) -- (1,0) -- (0,0) -- (-0.5,0.8) -- (1,0) -- (1.5,0.8) -- (0,0) -- (1,1.6) -- (0,1.6) -- (0,0) -- (-0.5,0.8) -- (0,1.6) -- (1,1.6) -- (-0.5,0.8) -- (1.5,0.8);
	        	\filldraw[black, fill=white] (0,0) circle (2pt);
	        	\filldraw[black, fill=white] (-0.5,0.8) circle (2pt);
	        	\filldraw[black, fill=white] (1,0) circle (2pt);
	        	\filldraw[black, fill=white] (1.5,0.8) circle (2pt);
	        	\filldraw[black, fill=white] (0,1.6) circle (2pt);
	        	\filldraw[black, fill=white] (1,1.6) circle (2pt);
	        	\filldraw (0.5,-0.2)node[below]{{\tiny$(4,2)$}};
	        	
	        	\draw[-] (2.5,0) -- (2.5,1.6);
	        	\draw[-] (4,0) -- (3,0) -- (3,1.6) -- (4,0) -- (4,1.6) -- (3,0) --(4,1.6) -- (3,1.6);
	        	\filldraw[black, fill=white] (2.5,0) circle (2pt);
	        	\filldraw[black, fill=white] (2.5,1.6) circle (2pt);
	        	\filldraw[black, fill=white] (3,0) circle (2pt);
	        	\filldraw[black, fill=white] (3,1.6) circle (2pt);
	        	\filldraw[black, fill=white] (4,0) circle (2pt);
	        	\filldraw[black, fill=white] (4,1.6) circle (2pt);
	        	\filldraw (3.3,-0.2)node[below]{{\tiny$(3,3)$}};
	        	
	        	\draw[-] (5.5,1.6) -- (7,0.8) -- (6.5,1.6) -- (6.5,0) -- (5.5,1.6) -- (5.5,0) -- (5,0.8) -- (5.5,1.6) -- (6.5,1.6) -- (6.5,0) -- (7,0.8);
	        	\filldraw[black, fill=white] (5.5,0) circle (2pt);
	        	\filldraw[black, fill=white] (5,0.8) circle (2pt);
	        	\filldraw[black, fill=white] (6.5,0) circle (2pt);
	        	\filldraw[black, fill=white] (7,0.8) circle (2pt);
	        	\filldraw[black, fill=white] (5.5,1.6) circle (2pt);
	        	\filldraw[black, fill=white] (6.5,1.6) circle (2pt);
	        	\filldraw (6,-0.2)node[below]{{\tiny$(4,3)$}};
	        	\draw[-] (8.5,1.6) -- (10,0.8) -- (9.5,1.6) -- (9.5,0) -- (8.5,1.6) -- (8.5,0) -- (8,0.8) -- (8.5,1.6) -- (9.5,1.6) -- (9.5,0) -- (10,0.8);
	        	\draw[-] (8,0.8) -- (9.5,1.6) -- (8.5,0);
	        	\filldraw[black, fill=white] (8.5,0) circle (2pt);
	        	\filldraw[black, fill=white] (8,0.8) circle (2pt);
	        	\filldraw[black, fill=white] (9.5,0) circle (2pt);
	        	\filldraw[black, fill=white] (10,0.8) circle (2pt);
	        	\filldraw[black, fill=white] (8.5,1.6) circle (2pt);
	        	\filldraw[black, fill=white] (9.5,1.6) circle (2pt);
	        	\filldraw (9,-0.2)node[below]{{\tiny$(5,3)$}};
	        	
	        	\draw[-] (11,0.8) -- (11.5,1.6) -- (13,0.8) -- (12.5,1.6) -- (12.5,0) -- (11.5,1.6) -- (11.5,0) -- (11.5,1.6) -- (12.5,1.6) -- (12.5,0) -- (13,0.8) -- (12.5,0) -- (11,0.8) -- (13,0.8);
	        	\draw[-] (11,0.8) -- (12.5,1.6) -- (11.5,0) -- (13,0.8);
	        	\filldraw[black, fill=white] (11.5,0) circle (2pt);
	        	\filldraw[black, fill=white] (11,0.8) circle (2pt);
	        	\filldraw[black, fill=white] (12.5,0) circle (2pt);
	        	\filldraw[black, fill=white] (13,0.8) circle (2pt);
	        	\filldraw[black, fill=white] (11.5,1.6) circle (2pt);
	        	\filldraw[black, fill=white] (12.5,1.6) circle (2pt);
	        	\filldraw (12,-0.2)node[below]{{\tiny$(6,3)$}};
	        	
	        	\draw[-] (14.5,1.6) -- (16,0.8) -- (15.5,1.6) -- (15.5,0) -- (14.5,1.6) -- (14.5,0) -- (14,0.8) -- (14.5,1.6) -- (15.5,1.6) -- (15.5,0) -- (16,0.8) -- (14,0.8) -- (14.5,0) -- (15.5,0) -- (15.5,1.6); 
	        	\draw[-] (15.5,1.6) -- (16,0.8) -- (14,0.8) -- (14.5,0) -- (15.5,1.6) -- (15.5,0) -- (14.5,0) -- (14.5,1.6) -- (14,0.8) -- (16,0.8) -- (14.5,0);
	        	\draw[-] (14,0.8) -- (15.5,1.6);
	        	\filldraw[black, fill=white] (14.5,0) circle (2pt);
	        	\filldraw[black, fill=white] (14,0.8) circle (2pt);
	        	\filldraw[black, fill=white] (15.5,0) circle (2pt);
	        	\filldraw[black, fill=white] (16,0.8) circle (2pt);
	        	\filldraw[black, fill=white] (14.5,1.6) circle (2pt);
	        	\filldraw[black, fill=white] (15.5,1.6) circle (2pt);
	        	\filldraw (15,-0.2)node[below]{{\tiny$(7,3)$}};
	        	\end{tikzpicture}
	        \end{center}
        \begin{center}
        	\begin{tikzpicture}[scale=0.83]
            \draw[-] (0.6,0) -- (0.6,0.8);
            \draw[-] (1.1,0) -- (1.1,0.8);
            \draw[-] (1.6,0) -- (1.6,0.8);
            \filldraw[black, fill=white] (1.6,0) circle (2pt);
            \filldraw[black, fill=white] (1.6,0.8) circle (2pt);
            \filldraw[black, fill=white] (0.6,0) circle (2pt);
            \filldraw[black, fill=white] (0.6,0.8) circle (2pt);
            \filldraw[black, fill=white] (1.1,0) circle (2pt);
            \filldraw[black, fill=white] (1.1,0.8) circle (2pt);
            \filldraw (1.1,-0.2)node[below]{{\tiny$(2,4)$}};
            
            \draw[-] (2.5,0) -- (3,0) -- (2.75,0.8) -- (2.5,0);
            \draw[-] (3.4,0) -- (4.3,0);
        	\filldraw[black, fill=white] (2.5,0) circle (2pt);
        	\filldraw[black, fill=white] (3,0) circle (2pt);
        	\filldraw[black, fill=white] (2.75,0.8) circle (2pt);
        	\filldraw[black, fill=white] (3.4,0) circle (2pt);
        	\filldraw[black, fill=white] (3.85,0) circle (2pt);
        	\filldraw[black, fill=white] (4.3,0) circle (2pt);
        	\filldraw (3.5,-0.2)node[below]{{\tiny$(3,4)$}};
        	
        	\draw[-] (6,0) -- (5.75,0.8) -- (5.5,0) -- (5.25,0.8) -- (5,0) -- (6.5,0);
        	\filldraw[black, fill=white] (5,0) circle (2pt);
        	\filldraw[black, fill=white] (5.5,0) circle (2pt);
        	\filldraw[black, fill=white] (6,0) circle (2pt);
        	\filldraw[black, fill=white] (6.5,0) circle (2pt);
        	\filldraw[black, fill=white] (5.25,0.8) circle (2pt);
        	\filldraw[black, fill=white] (5.75,0.8) circle (2pt);
        	\filldraw (5.75,-0.2)node[below]{{\tiny$(4,4)$}};
        	
        	\draw[-] (7.8,0.8) -- (8.05,0) -- (8.3,0.8) -- (8.05,0) -- (8.8,0.8) -- (8.55,1.6) -- (8.3,0.8) --  (7.3,0.8) -- (8.05,0) -- (7.3,0.8) -- (8.8,0.8);
        	\filldraw[black, fill=white] (7.3,0.8) circle (2pt);
        	\filldraw[black, fill=white] (7.8,0.8) circle (2pt);
        	\filldraw[black, fill=white] (8.3,0.8) circle (2pt);
        	\filldraw[black, fill=white] (8.8,0.8) circle (2pt);
        	\filldraw[black, fill=white] (8.05,0) circle (2pt);
        	\filldraw[black, fill=white] (8.55,1.6) circle (2pt);
        	\filldraw (8.05,-0.2)node[below]{{\tiny$(5,4)$}};
        	
        	\draw[-] (9.6,0.8) -- (10.35,0) -- (10.1,0.8) -- (10.35,0) -- (10.6,0.8) -- (10.6,1.6) -- (10.1,0.8) -- (10.6,1.6) -- (11.1,0.8) -- (9.6,0.8) -- (10.35,0) -- (11.1,0.8) -- (10.35,0);
        	\filldraw[black, fill=white] (9.6,0.8) circle (2pt);
        	\filldraw[black, fill=white] (10.1,0.8) circle (2pt);
        	\filldraw[black, fill=white] (10.6,0.8) circle (2pt);
        	\filldraw[black, fill=white] (11.1,0.8) circle (2pt);
        	\filldraw[black, fill=white] (10.35,0) circle (2pt);
        	\filldraw[black, fill=white] (10.6,1.6) circle (2pt);
        	\filldraw (10.35,-0.2)node[below]{{\tiny$(6,4)$}};
        	
        	\draw[-] (11.9,0.8) -- (12.65,0) -- (12.9,0.8) -- (12.65,0) -- (13.4,0.8) -- (12.65,1.6) -- (12.9,0.8) -- (12.65,1.6) -- (12.4,0.8) -- (12.65,1.6) -- (13.4,0.8) -- (11.9,0.8) -- (12.65,0) -- (13.4,0.8) -- (12.65,0) -- (12.4,0.8)  -- (12.65,1.6) -- (11.9,0.8);
        	\filldraw[black, fill=white] (11.9,0.8) circle (2pt);
        	\filldraw[black, fill=white] (12.4,0.8) circle (2pt);
        	\filldraw[black, fill=white] (12.9,0.8) circle (2pt);
        	\filldraw[black, fill=white] (13.4,0.8) circle (2pt);
        	\filldraw[black, fill=white] (12.65,0) circle (2pt);
        	\filldraw[black, fill=white] (12.65,1.6) circle (2pt);
        	\filldraw (12.65,-0.2)node[below]{{\tiny$(7,4)$}};
        	\end{tikzpicture}
        \end{center}\smallskip
        \begin{center}
        	\begin{tikzpicture}[scale=0.82]
        	\draw[-] (2.5,0) -- (2.5,1.2);
        	\draw[-] (3.5,0) -- (3.5,1.2);
        	\filldraw[black, fill=white] (2.5,0) circle (2pt);
        	\filldraw[black, fill=white] (2.5,0.6) circle (2pt);
        	\filldraw[black, fill=white] (2.5,1.2) circle (2pt);
        	\filldraw[black, fill=white] (3.5,0) circle (2pt);
            \filldraw[black, fill=white] (3.5,0.6) circle (2pt);
            \filldraw[black, fill=white] (3.5,1.2) circle (2pt);
        	\filldraw (3,-0.2)node[below]{{\tiny$(3,5)$}};
        	\draw[-] (5,0) -- (4.75,0.8) -- (4.5,0) -- (6.5,0);
        	\filldraw[black, fill=white] (4.5,0) circle (2pt);
        	\filldraw[black, fill=white] (5,0) circle (2pt);
        	\filldraw[black, fill=white] (5.5,0) circle (2pt);
        	\filldraw[black, fill=white] (6,0) circle (2pt);
        	\filldraw[black, fill=white] (6.5,0) circle (2pt);
        	\filldraw[black, fill=white] (4.75,0.8) circle (2pt);
        	\filldraw (5.5,-0.2)node[below]{{\tiny$(4,5)$}};
        	\draw[-] (7.5,0) -- (9.5,0) -- (8.5,0.8) -- (7.5,0) -- (8,0) -- (8.5,0.8) -- (8.5,0) -- (9,0) -- (8.5,0.8);
        	\filldraw[black, fill=white] (7.5,0) circle (2pt);
        	\filldraw[black, fill=white] (8,0) circle (2pt);
        	\filldraw[black, fill=white] (8.5,0) circle (2pt);
        	\filldraw[black, fill=white] (9,0) circle (2pt);
        	\filldraw[black, fill=white] (9.5,0) circle (2pt);
        	\filldraw[black, fill=white] (8.5,0.8) circle (2pt);
        	\filldraw (8.5,-0.2)node[below]{{\tiny$(5,5)$}};
        	\draw[-] (10.5,0) -- (13,0);
        	\filldraw[black, fill=white] (10.5,0) circle (2pt);
        	\filldraw[black, fill=white] (11,0) circle (2pt);
        	\filldraw[black, fill=white] (11.5,0) circle (2pt);
        	\filldraw[black, fill=white] (12,0) circle (2pt);
        	\filldraw[black, fill=white] (12.5,0) circle (2pt);
        	\filldraw[black, fill=white] (13,0) circle (2pt);
        	\filldraw (11.75,-0.2)node[below]{{\tiny$(4,6)$}};
        	\end{tikzpicture}
        \end{center}\medskip
        Note that all graphs $G$ displayed, with $\pd(J_G)\ge n-2=4$, are connected. Furthermore, the graphs with regularity 3 are constructed as in Corollary \ref{Cor:reg=3}. Whereas, the graphs with regularity $n-2=4$ are constructed by using Proposition \ref{Prop:reg=n-2}. Finally, if $\reg(J_G)=n-1$, then $\pd(J_G)\le 2n-7=5$ by Corollary \ref{Cor:pd=2n-5} and Proposition \ref{Prop:pd=2n-6}.
		\end{enumerate}
	\end{Example}

	\section{The size of the betti table of binomial edge ideals}\label{FS22:sec3}
	Let $n\ge1$ be an integer. As before, denote by $\textup{Graphs}(n)$ the class of all finite simple graphs on $n$ non-isolated vertices, and let
	$$
	\textup{pdreg}(n)\ =\ \big\{(\pd(J_G),\reg(J_G)):G\in\textup{Graphs}(n)\big\},
	$$
	be the set of the sizes of the Betti tables of $J_G\subset S=K[x_1,\dots,x_n,y_1,\dots,y_n]$, as $G$ ranges over all graphs on $n$ non-isolated vertices.

    Note that we are allowing $K$ to be any field.\smallskip
	
	Finally, we can state our main result in the article.
	
	\begin{Theorem}\label{Thm:pdreg(n)}
		For all $n\ge3$,
		\begin{equation}\label{eq:pdregSet}
		\begin{aligned}
			\textup{pdreg}(n)\ =\ \big\{(n-2,2),(n-2,n)\big\}\cup\bigcup_{r=3}^{\lfloor\frac{n}{2}\rfloor+1}\big(\bigcup_{p=n-r}^{2n-5}\{(p,r)\}\big)\ \cup\ \\
		\cup \bigcup_{r=\lceil\frac{n}{2}\rceil+1}^{n-2}\big(\bigcup_{p=r-2}^{2n-5}\{(p,r)\}\big)\cup A_n,
		\end{aligned}
		\end{equation}
		where $A_n=\{(p,r)\in\textup{pdreg}(n):r=n-1\}$.
	\end{Theorem}\smallskip

    The following picture describes the set $\textup{pdreg}(n)$ for $n\ge3$. In the $(p,r)$th position of the diagram we collocate the pair $(p,r)$ if there exists $G\in\textup{Graphs}(n)$ such that $\pd(J_G)=p$ and $\reg(J_G)=r$.
    \begin{equation*}
         \begin{tikzpicture}
	\begin{pgfonlayer}{nodelayer}
		\node [style=none, scale=0.6] (0) at (-52, -0.25) {};
		\node [style=none, scale=0.6] (1) at (-52, -14.25) {};
		\node [style=none, scale=0.6] (2) at (-33.25, -14.25) {};
		\node [style=none, scale=0.6] (4) at (-52.5, -12.25) {$2$};
		\node [style=none, scale=0.6] (6) at (-52.5, -11.25) {$3$};
		\node [style=none, scale=0.6] (7) at (-52.5, -1.25) {$n$};
		\node [style=none, scale=0.6] (12) at (-35, -14.75) {$\ \ \ 2n-5$};
		\node [style=none, scale=0.6] (13) at (-36.25, -14.75) {  $2n-6\ \ $  };
		\node [style=full, scale=0.3] (35) at (-36, -11.25) {};
		\node [style=full, scale=0.3] (36) at (-35, -11.25) {};
		\node [style=full, scale=0.3] (42) at (-40, -11.25) {};
		\node [style=full, scale=0.3] (43) at (-39, -11.25) {};
		\node [style=full, scale=0.3] (44) at (-38, -11.25) {};
		\node [style=full, scale=0.3] (45) at (-37, -11.25) {};
		\node [style=full, scale=0.3] (80) at (-36, -10.25) {};
		\node [style=full, scale=0.3] (81) at (-35, -10.25) {};
		\node [style=full, scale=0.3] (87) at (-40, -10.25) {};
		\node [style=full, scale=0.3] (88) at (-39, -10.25) {};
		\node [style=full, scale=0.3] (89) at (-38, -10.25) {};
		\node [style=full, scale=0.3] (90) at (-37, -10.25) {};
		\node [style=full, scale=0.3] (125) at (-36, -7.25) {};
		\node [style=full, scale=0.3] (126) at (-35, -7.25) {};
		\node [style=full, scale=0.3] (128) at (-50, -7.25) {};
		\node [style=full, scale=0.3] (129) at (-49, -7.25) {};
		\node [style=full, scale=0.3] (130) at (-48, -7.25) {};
		\node [style=full, scale=0.3] (132) at (-40, -7.25) {};
		\node [style=full, scale=0.3] (133) at (-39, -7.25) {};
		\node [style=full, scale=0.3] (134) at (-38, -7.25) {};
		\node [style=full, scale=0.3] (135) at (-37, -7.25) {};
		\node [style=full, scale=0.3] (140) at (-36, -6.25) {};
		\node [style=full, scale=0.3] (141) at (-35, -6.25) {};
		\node [style=full, scale=0.3] (143) at (-50, -6.25) {};
		\node [style=full, scale=0.3] (144) at (-49, -6.25) {};
		\node [style=full, scale=0.3] (145) at (-48, -6.25) {};
		\node [style=full, scale=0.3] (147) at (-40, -6.25) {};
		\node [style=full, scale=0.3] (148) at (-39, -6.25) {};
		\node [style=full, scale=0.3] (149) at (-38, -6.25) {};
		\node [style=full, scale=0.3] (150) at (-37, -6.25) {};
		\node [style=full, scale=0.3] (185) at (-36, -3.25) {};
		\node [style=full, scale=0.3] (186) at (-35, -3.25) {};
		\node [style=full, scale=0.3] (190) at (-47, -3.25) {};
		\node [style=full, scale=0.3] (192) at (-40, -3.25) {};
		\node [style=full, scale=0.3] (193) at (-39, -3.25) {};
		\node [style=full, scale=0.3] (194) at (-38, -3.25) {};
		\node [style=full, scale=0.3] (195) at (-37, -3.25) {};
		\node [style=none, scale=0.6] (196) at (-43, -14.75) {$n$};
		\node [style=none, scale=0.6] (197) at (-44, -14.75) {$\ n-1$};
		\node [style=none, scale=0.6] (198) at (-45.25, -14.75) {$\ n-2\ \ $};
		\node [style=full, scale=0.3] (199) at (-45, -12.25) {};
		\node [style=full, scale=0.3] (200) at (-45, -1.25) {};
		\node [style=none, scale=0.6] (205) at (-46.5, -7.25) {   $\cdots$   };
		\node [style=none, scale=0.6] (206) at (-46.5, -6.25) {   $\cdots$   };
		\node [style=none, scale=0.6] (210) at (-41.5, -11.25) {   $\cdots$   };
		\node [style=none, scale=0.6] (211) at (-41.5, -10.25) {   $\cdots$   };
		\node [style=none, scale=0.6] (214) at (-41.5, -7.25) {   $\cdots$   };
		\node [style=none, scale=0.6] (215) at (-41.5, -6.25) {   $\cdots$   };
		\node [style=none, scale=0.6] (218) at (-41.5, -3.25) {   $\cdots$   };
		\node [style=full, scale=0.3] (219) at (-45, -11.25) {};
		\node [style=full, scale=0.3] (220) at (-45, -10.25) {};
		\node [style=full, scale=0.3] (223) at (-45, -7.25) {};
		\node [style=full, scale=0.3] (224) at (-45, -6.25) {};
		\node [style=full, scale=0.3] (227) at (-45, -3.25) {};
		\node [style=full, scale=0.3] (228) at (-44, -11.25) {};
		\node [style=full, scale=0.3] (229) at (-44, -10.25) {};
		\node [style=full, scale=0.3] (232) at (-44, -7.25) {};
		\node [style=full, scale=0.3] (233) at (-44, -6.25) {};
		\node [style=full, scale=0.3] (236) at (-44, -3.25) {};
		\node [style=full, scale=0.3] (237) at (-43, -11.25) {};
		\node [style=full, scale=0.3] (238) at (-43, -10.25) {};
		\node [style=full, scale=0.3] (241) at (-43, -7.25) {};
		\node [style=full, scale=0.3] (242) at (-43, -6.25) {};
		\node [style=full, scale=0.3] (245) at (-43, -3.25) {};
		\node [style=none, scale=0.6] (246) at (-52.5, -2.25) {$n-1\ \ \ \ \ \ $};
		\node [style=none, scale=0.6] (247) at (-52.5, -6.25) {$\lceil\frac{n}{2}\rceil+1\ \ \ \ \ \ \ \ $};
		\node [style=empty, scale=0.3] (263) at (-45, -2.25) {};
		\node [style=empty, scale=0.3] (264) at (-44, -2.25) {};
		\node [style=empty, scale=0.3] (265) at (-43, -2.25) {};
		\node [style=none, scale=0.6] (266) at (-52.5, -7.25) {$\lfloor\frac{n}{2}\rfloor+1\ \ \ \ \ \ \ \ $};
		\node [style=none, scale=0.6] (273) at (-39, -4.25) {  $\vdots$  };
		\node [style=none, scale=0.6] (274) at (-40, -4.25) {  $\vdots$  };
		\node [style=none, scale=0.6] (275) at (-37, -4.25) {  $\vdots$  };
		\node [style=none, scale=0.6] (276) at (-38, -4.25) {  $\vdots$  };
		\node [style=none, scale=0.6] (277) at (-35, -4.25) {  $\vdots$  };
		\node [style=none, scale=0.6] (278) at (-36, -4.25) {  $\vdots$  };
		\node [style=none, scale=0.6] (279) at (-46.5, -4.25) {  $\vdots$  };
		\node [style=none, scale=0.6] (280) at (-48, -4.25) {\rotatebox{90}{$\ddots$}};
		\node [style=none, scale=0.6] (281) at (-41.5, -4.25) {  $\vdots$  };
		\node [style=none, scale=0.6] (282) at (-46.5, -9.25) {  $\vdots$  };
		\node [style=none, scale=0.6] (283) at (-41.5, -9.25) {  $\vdots$  };
		\node [style=none, scale=0.6] (297) at (-45, -9.25) {  $\vdots$  };
		\node [style=none, scale=0.6] (298) at (-44, -9.25) {  $\vdots$  };
		\node [style=none, scale=0.6] (299) at (-43, -9.25) {  $\vdots$  };
		\node [style=none, scale=0.6] (300) at (-45, -4.25) {  $\vdots$  };
		\node [style=none, scale=0.6] (301) at (-44, -4.25) {  $\vdots$  };
		\node [style=none, scale=0.6] (302) at (-43, -4.25) {  $\vdots$  };
		\node [style=none, scale=0.6] (309) at (-48, -9.25) { \rotatebox{-5}{$\ddots$} };
		\node [style=none, scale=0.6] (310) at (-39, -9.25) {  $\vdots$  };
		\node [style=none, scale=0.6] (311) at (-40, -9.25) {  $\vdots$  };
		\node [style=none, scale=0.6] (312) at (-37, -9.25) {  $\vdots$  };
		\node [style=none, scale=0.6] (313) at (-38, -9.25) {  $\vdots$  };
		\node [style=none, scale=0.6] (314) at (-35, -9.25) {  $\vdots$  };
		\node [style=none, scale=0.6] (315) at (-36, -9.25) {  $\vdots$  };
		\node [style=none, scale=0.6] (320) at (-52.5, -10.25) {$4$};
		\node [style=none, scale=0.6] (325) at (-41.5, -14.75) {   $\cdots$   };
		\node [style=none, scale=0.6] (326) at (-38.75, -14.75) {   $\cdots$   };
		\node [style=none, scale=0.6] (329) at (-52.5, -8.75) {   $\vdots$   };
		\node [style=none, scale=0.6] (330) at (-52.5, -4.75) {   $\vdots$   };
		\node [style=none, scale=0.6] (331) at (-46.5, -8.25) {  $\vdots$  };
		\node [style=none, scale=0.6] (332) at (-41.5, -8.25) {  $\vdots$  };
		\node [style=none, scale=0.6] (333) at (-45, -8.25) {  $\vdots$  };
		\node [style=none, scale=0.6] (334) at (-44, -8.25) {  $\vdots$  };
		\node [style=none, scale=0.6] (335) at (-43, -8.25) {  $\vdots$  };
		\node [style=none, scale=0.6] (340) at (-49, -8.25) { \rotatebox{-5}{$\ddots$} };
		\node [style=none, scale=0.6] (342) at (-48, -8.25) { \rotatebox{-5}{$\ddots$} };
		\node [style=none, scale=0.6] (343) at (-39, -8.25) {  $\vdots$  };
		\node [style=none, scale=0.6] (344) at (-40, -8.25) {  $\vdots$  };
		\node [style=none, scale=0.6] (345) at (-37, -8.25) {  $\vdots$  };
		\node [style=none, scale=0.6] (346) at (-38, -8.25) {  $\vdots$  };
		\node [style=none, scale=0.6] (347) at (-35, -8.25) {  $\vdots$  };
		\node [style=none, scale=0.6] (348) at (-36, -8.25) {  $\vdots$  };
		\node [style=none, scale=0.6] (349) at (-46.5, -5.25) {  $\vdots$  };
		\node [style=none, scale=0.6] (350) at (-41.5, -5.25) {  $\vdots$  };
		\node [style=none, scale=0.6] (351) at (-45, -5.25) {  $\vdots$  };
		\node [style=none, scale=0.6] (352) at (-44, -5.25) {  $\vdots$  };
		\node [style=none, scale=0.6] (353) at (-43, -5.25) {  $\vdots$  };
		\node [style=none, scale=0.6] (358) at (-49, -5.25) {\rotatebox{90}{$\ddots$}};
		\node [style=none, scale=0.6] (360) at (-48, -5.25) {\rotatebox{90}{$\ddots$}};
		\node [style=none, scale=0.6] (361) at (-39, -5.25) {  $\vdots$  };
		\node [style=none, scale=0.6] (362) at (-40, -5.25) {  $\vdots$  };
		\node [style=none, scale=0.6] (363) at (-37, -5.25) {  $\vdots$  };
		\node [style=none, scale=0.6] (364) at (-38, -5.25) {  $\vdots$  };
		\node [style=none, scale=0.6] (365) at (-35, -5.25) {  $\vdots$  };
		\node [style=none, scale=0.6] (366) at (-36, -5.25) {  $\vdots$  };
		\node [style=none, scale=0.6] (367) at (-52.5, -3.25) {$n-2\ \ \ \ \ \ $};
		\node [style=full, scale=0.3] (368) at (-46, -11.25) {};
		\node [style=full, scale=0.3] (369) at (-47, -10.25) {};
		\node [style=full, scale=0.3] (371) at (-46, -10.25) {};
		\node [style=none, scale=0.6] (372) at (-50, -14.75) {$\lceil\frac{n}{2}\rceil-1\ $};
		\node [style=none, scale=0.6] (373) at (-48.25, -14.75) {   $\ \ \ \ \cdots$   };
		\node [style=empty, scale=0.3] (374) at (-46, -2.25) {};
		\node [style=full, scale=0.3] (375) at (-46, -3.25) {};
		\node [style=none, scale=0.6] (376) at (-41.5, -2.25) {   $\cdots$   };
		\node [style=empty, scale=0.3] (377) at (-39, -2.25) {};
		\node [style=empty, scale=0.3] (378) at (-38, -2.25) {};
		\node [style=empty, scale=0.3] (379) at (-37, -2.25) {};
		\node [style=empty, scale=0.3] (380) at (-40, -2.25) {};
		\node [style=empty, scale=0.3] (381) at (-35, -2.25) {};
		\node [style=empty, scale=0.3] (384) at (-36, -2.25) {};
		\node [style=none] (388) at (-45, -14.25) {\tiny{$\vert$}};
		\node [style=none] (389) at (-43, -14.25) {\tiny{$\vert$}};
		\node [style=none] (390) at (-36, -14.25) {\tiny{$\vert$}};
		\node [style=none] (391) at (-35, -14.25) {\tiny{$\vert$}};
		\node [style=none] (392) at (-50, -14.25) {\tiny{$\vert$}};
		\node [style=none] (393) at (-44, -14.25) {\tiny{$\vert$}};
		\node [style=none] (395) at (-52, -12.25) {\rotatebox{90}{\tiny{$\vert$}}};
		\node [style=none] (396) at (-52, -11.25) {\rotatebox{90}{\tiny{$\vert$}}};
		\node [style=none] (397) at (-52, -10.25) {\rotatebox{90}{\tiny{$\vert$}}};
		\node [style=none] (398) at (-52, -7.25) {\rotatebox{90}{\tiny{$\vert$}}};
		\node [style=none] (399) at (-52, -6.25) {\rotatebox{90}{\tiny{$\vert$}}};
		\node [style=none] (400) at (-52, -3.25) {\rotatebox{90}{\tiny{$\vert$}}};
		\node [style=none] (401) at (-52, -2.25) {\rotatebox{90}{\tiny{$\vert$}}};
		\node [style=none] (402) at (-52, -1.25) {\rotatebox{90}{\tiny{$\vert$}}};
	\end{pgfonlayer}
	\begin{pgfonlayer}{edgelayer}
		\draw (0.center) to (1.center);
		\draw [in=180, out=0] (1.center) to (2.center);
	\end{pgfonlayer}
\end{tikzpicture}

    \end{equation*}\medskip
    
    Note that the lattice points in the $(n-1)$th row are empty, because we do not specify the set $A_n$ in Theorem \ref{Thm:pdreg(n)}.
    
	Now, we are ready to prove our main result in the article.
	\begin{proof}[Proof of Theorem \ref{Thm:pdreg(n)}.]
		Let $(p,r)\in\textup{pdreg}(n)$. By Theorem \ref{Thm:regGbounds}, $2\le r\le n$, and $(p,2),(p,n)\in\textup{pdreg}(n)$ if and only if $p=n-2$. Furthermore, by Propositions \ref{Prop:refinedPd1} and \ref{Prop:refinedPd2}, if $3\le r\le\lfloor\frac{n}{2}\rfloor+1$ then $n-r\le p\le2n-5$, and if $\lceil\frac{n}{2}\rceil+1\le r\le n-2$ then $r-2\le p\le 2n-5$. Hence, we see that the set $\textup{pdreg}(n)$ is contained in the second set written in (\ref{eq:pdregSet}). Therefore, we only need to prove the other inclusion.
		
		We proceed by induction on $n\ge3$. By induction we also prove the following\medskip\\
		\textbf{Claim $(*)$}\ \ For all pairs $(p,r)\in\textup{pdreg}(n)\setminus A_n$ with $p\ge n-2$, there exists a connected graph $G\in\textup{Graphs}(n)$ such that $(\pd(J_G),\reg(J_G))=(p,r)$.\medskip
		
		Now, we start with our inductive proof. If $n=3,4,5,6$, by Example \ref{Ex:InitialCasesPdReg}, the theorem and the \textbf{Claim $(*)$} hold true.
		
		Suppose now $n\ge7$. Let $(p,r)\in\textup{pdreg}(n)$. Firstly we prove the \textbf{Claim $(*)$}.
		\medskip\\
		\textit{Proof of} \textbf{Claim $(*)$}. We distinguish several cases.
		\medskip\\
		{\bf Case} ${\bf r=2}$ or ${\bf r=n}$. Then $p=n-2$ and the pairs $(n-2,2)$, $(n-2,n)$ belongs to $\textup{pdreg}(n)$, by virtue of Theorem \ref{Thm:regGbounds}(i)-(ii).
		\medskip\\
		{\bf Case} ${\bf 3\le r\le n-2}$. If $p=2n-5,2n-6$, then $3\le r\le n-2$ and $(p,r)\in\textup{pdreg}(n)$ for all such values of $p$ and $r$, by using Corollaries \ref{Cor:pd=2n-5} and \ref{Cor:pd=2n-6} and Proposition \ref{Prop:pd=2n-6}. If $r=3$ or $r=n-2$, all possible pairs $(p,3)$ and $(p,n-2)$ belong to $\textup{pdreg}(n)$, by Corollary \ref{Cor:reg=3} and Proposition \ref{Prop:reg=n-2}.
		
		It remains to construct $G\in\textup{Graphs}(n)$ such that $(\pd(J_G),\reg(J_G))=(p,r)$ for all $4\le r\le n-3$ and all admissible values that $n-2\le p\le 2n-7$ can assume.
		
		Suppose $n-3\le p\le 2n-7$ and $4\le r\le n-3$. Then $(p-2,r)\in\textup{pdreg}(n-1)$. To prove this, note that $(n-1)-3\le p-2\le 2(n-1)-7$ and $4\le r\le (n-1)-2$. By induction there exists a connected graph $\widetilde{G}\in\textup{Graphs}(n-1)$ such that
		$$
		(\pd(J_{\widetilde{G}}),\reg(J_{\widetilde{G}}))=(p-2,r).
		$$
		Set $G=\widetilde{G}*K_1$. By Lemma \ref{Lemma:KumarSarkar} $\pd(J_G)=\pd(J_{\widetilde{G}})+2=p$ and by formula (\ref{eq:RegJoinGraphs}), $\reg(J_G)=\reg(J_{\widetilde{G}})=r$. Hence $(p,r)\in\textup{pdreg}(n)$.
		
		Suppose now $p=n-2$ and $4\le r\le n-3$. We distinguish two more cases.
		\medskip\\
		{\bf Case} ${\bf p=n-2,\,\,r=n-3}$. Let $G_1=P_{n-4}$ and let $G_2=K_3$ be the complete graph on vertex set $\{n-3,n-2,n-1\}$. Let $\widetilde{G}$ be the disjoint union of $G_1,G_2$. Then $\widetilde{G}\in\textup{Graphs}(n-1)$ and by Theorem \ref{Thm:regGbounds}(i)-(ii) and Remark \ref{Rem:Gdecomp},
		$$
		(\pd(J_{\widetilde{G}}),\reg(J_{\widetilde{G}}))=((n-4)-2+(3-2)+1,n-4+2-1)=(n-4,n-3).
		$$
		Let $G=\widetilde{G}*K_1$. Since $\widetilde{G}$ is disconnected and $\reg(J_{\widetilde{G}})=n-3\ge3$, Lemma \ref{Lemma:KumarSarkar} and the formula (\ref{eq:RegJoinGraphs}) yield that
		$$
		(\pd(J_G),\reg(J_G))=(\max\{n-4+2,n-3\},n-3)=(n-2,n-3).
		$$
		Hence $(n-2,n-3)\in\textup{pdreg}(n)$.
		\medskip\\
		\textbf{Case} ${\bf p=n-2,\,\,4\le r\le n-4}$. Then $p-3=n-5=(n-3)-2$ and $3\le r-1\le(n-3)-2$. Thus, by induction there exists a graph $G_1\in\textup{Graphs}(n-3)$ such that $\pd(J_{G_1})=p-3$ and $\reg(J_{G_1})=r-1$. Let $\widetilde{G}$ be the disjoint union of $G_1$ and the edge $\{n-2,n-1\}$. By Remark \ref{Rem:Gdecomp} we have $\pd(J_{\widetilde{G}})=p-2$ and $\reg(J_{\widetilde{G}})=r$. Let $G=\widetilde{G}*K_1$. By Lemma \ref{Lemma:KumarSarkar},
		$$
		\pd(J_G)=\max\{\pd(J_{\widetilde{G}})+2,n-3\}=\max\{p+2,n-3\}=n-2=p
		$$
		and by formula (\ref{eq:RegJoinGraphs}), $\reg(J_G)=\reg(J_{\widetilde{G}})=r$. Hence, $G\in\textup{Graphs}(n)$,
		$$
		(\pd(J_G),\reg(J_G))=(n-2,r),
		$$
		and so $(n-2,r)\in\textup{pdreg}(n)$.
		
		Now the inductive proof of the \textbf{Claim $(*)$} is completed. Indeed, by Remarks \ref{Rem:pd=2n-5,2n-6,Connected}, \ref{Rem:reg=3,GConnected}, \ref{Rem:Gconnectr=n-2} the claim holds for all pairs $(p,r)\in\textup{pdreg}(n)\setminus A_n$, $p\ge n-2$, with $p=2n-6$ or $p=2n-r$ or $r=3$ or $r=n-2$. For all other pairs $(p,r)\in\textup{pdreg}(n)\setminus A_n$ with $p\ge n-2$, the claim also holds because the various graphs $\widetilde{G}*K_1$ constructed are connected.\hfill$\square$\smallskip
		
		Having acquired {\bf Claim $(*)$}, we prove the theorem. Let $n\ge7$ and let $(p,r)\in\textup{pdreg}(n)$. Depending on the values of $r$ we distinguish several cases.
		\medskip\\
		\textbf{Case} ${\bf r=2}$ or ${\bf r=n}$. Then $p=n-2$ and $(n-2,2),(n-2,n)\in\textup{pdreg}(n)$ by Theorem \ref{Thm:regGbounds}(i)-(ii).
		\medskip\\
		{\bf Case ${\bf r=3}$}. Then $n-3\le p\le 2n-5$ by Proposition \ref{Prop:refinedPd1} and for any such $p$, $(p,r)\in\textup{pdreg}(n)$ by Corollary \ref{Cor:reg=3}.
		\medskip\\
		{\bf Case ${\bf 4\le r\le\lfloor\frac{n}{2}\rfloor+1}$}. If $n-2\le p\le 2n-5$, the pairs $(p,r)$ belong to $\textup{pdreg}(n)$ by the {\bf Claim $(*)$}. Suppose $n-r\le p\le n-3$. Then $(p-1,r-1)$ belongs to $\textup{pdreg}(n-2)$. To see why this is true, note that $3\le r-1\le\lfloor\frac{n-2}{2}\rfloor+1$ and also $(n-2)-(r-1)=n-1-r\le p-1\le(n-2)-2$. Therefore, by induction there exists $\widetilde{G}\in\textup{Graphs}(n-2)$ such that
		$$
		(\pd(J_{\widetilde{G}}),\reg(J_{\widetilde{G}}))=(p-1,r-1).
		$$
		Set $G=\widetilde{G}\sqcup K_2$. Then $G\in\textup{Graphs}(n)$ and
		\begin{align*}
		(\pd(J_G),\reg(J_G))&=(\pd(J_{\widetilde{G}})+\pd(J_{K_2})+1,\reg(J_{\widetilde{G}})+\reg(J_{K_2})-1)\\
		&=(p-1+0+1,r-1+2-1)=(p,r).
		\end{align*}
		Consequently, $(p,r)\in\textup{pdreg}(n)$.
		\medskip\\
		{\bf Case ${\bf\lceil\frac{n}{2}\rceil+1\le r\le n-3}$}. If $n-2\le p\le 2n-5$, the pairs $(p,r)$ belong to $\textup{pdreg}(n)$ by the {\bf Claim $(*)$}. Suppose $r-2\le p\le n-3$. Then $(p-1,r-1)\in\textup{pdreg}(n-2)$. Indeed, $\lceil\frac{n-2}{2}\rceil+1\le r-1\le(n-2)-2$ and so $(r-1)-2\le p-1\le (n-2)-2$. Hence, by induction there exists $\widetilde{G}\in\textup{Graphs}(n-2)$ such that
		$$
		(\pd(J_{\widetilde{G}}),\reg(J_{\widetilde{G}}))=(p-1,r-1).
		$$
		Let $G=\widetilde{G}\sqcup K_2$. Arguing as before, we obtain $(p,r)\in\textup{pdreg}(n)$, as desired.
		\medskip\\
		{\bf Case} ${\bf r=n-2}$. In this case $n-4\le p\le 2n-5$ and all pairs $(p,n-2)\in\textup{pdreg}(n)$ by virtue of Proposition \ref{Prop:reg=n-2}.
		
		The inductive proof is complete, and the theorem is proved.
	\end{proof}

	Denote by $\textup{CGraphs}(n)$ the set of all connected graphs on $n$ non-isolated vertices. We define the set
	$$
	\textup{pdreg}_{\textup{C}}(n)\ =\ \big\{(\pd(J_G),\reg(J_G)):G\in\textup{CGraphs}(n)\big\}.
	$$
	\begin{Corollary}
		For all $n\ge3$,
		$$
		\textup{pdreg}_{\textup{C}}(n)\ =\ \big\{(n-2,2),(n-2,n)\big\}\cup\bigcup_{r=3}^{n-2}\big(\bigcup_{p=n-2}^{2n-5}\{(p,r)\}\big)\cup A_{\textup{C},n},
		$$
		where $A_{\textup{C},n}=\{(p,r)\in\textup{pdreg}_{\textup{C}}(n):r=n-1\}$.
	\end{Corollary}
	\begin{proof}
		For any $G\in\textup{CGraphs}(n)$ we have $\pd(J_G)\ge n-2$ by Theorem \ref{Thm:pdGconnect}. Thus, our statement follows from \textbf{Claim $(*)$} proved in the previous theorem.
	\end{proof}

	At present, we do not know yet the sets $A_n$ and $A_{\textup{C},n}$, for $n\ge7$. Note that by Corollary \ref{Cor:pd=2n-5} and Proposition \ref{Prop:pd=2n-6}, for all $n\ge 6$, if $G$ is a graph on $n$ non-isolated vertices and with $\reg(J_G)=n-1$, then $\pd(J_G)\le 2n-7$. On the other hand, if $n\ge7$ and $\reg(J_G)=n-1$, a much stronger bound for the projective dimension of $J_G$ seems to hold. Indeed, our experiments using \cite{MPNauty} suggest the following
	
	\begin{Conjecture}\label{Conj:reg=n-1}
		\rm Let $G$ be a graph on $n\ge7$ non-isolated vertices. Suppose that $\reg(J_G)=n-1$. Then $\pd(J_G)\le n$.
	\end{Conjecture}

	Using \cite{MPNauty} we could verify our conjecture for $n=7,8,9$.\bigskip
	
	\noindent\textbf{Acknowledgment.} The authors thank the anonymous referee for his/her careful
	reading and helpful suggestions, that allowed us to improve the quality of the paper.
	The authors acknowledge support of the GNSAGA of INdAM (Italy).
	
	A. Ficarra was partly supported by the Grant JDC2023-051705-I funded by MICIU/AEI/10.13039/501100011033 and by the FSE+.

\end{document}